\documentclass[12pt, a4paper, parskip=half, abstracton, bibliography=totoc]{scrartcl}
\pdfoutput=1
\usepackage{array}
\usepackage{marginnote}
\usepackage[dvipsnames]{xcolor}
\usepackage{tikz-cd}
\usepackage{hyperref}
\usepackage{amscd,amssymb,amsfonts,amsmath,latexsym,amsthm}
\usepackage[capitalize]{cleveref}
\usepackage[all,cmtip]{xy}
\usepackage{mathrsfs}
\usepackage{graphicx}
\usepackage{bm} 
\usepackage{mathtools} 

\usepackage{etoolbox}

\let\counterwithout\relax

\usepackage{chngcntr} 

\usepackage{defs_pp1}
\usepackage{slashed}
\usepackage[utf8]{inputenc}
\usepackage{microtype}
\usepackage[english]{babel}

\title{Homotopy theory with marked additive categories}
\date{\today}
\author{
Ulrich Bunke\thanks{Fakult{\"a}t f{\"u}r Mathematik,
Universit{\"a}t Regensburg,
93040 Regensburg,
Germany\newline
ulrich.bunke@mathematik.uni-regensburg.de} 
\and Alexander Engel\thanks{Fakult{\"a}t f{\"u}r Mathematik,
Universit{\"a}t Regensburg,
93040 Regensburg,
Germany\newline
alexander.engel@mathematik.uni-regensburg.de
	}
\and
Daniel Kasprowski\thanks{
	Rheinische Friedrich-Wilhelms-Universit\"at Bonn, Mathematisches Institut, Endenicher Allee 60,\newline 53115 Bonn, Germany\newline
	kasprowski@uni-bonn.de
}
\and
Christoph Winges\thanks{
		Rheinische Friedrich-Wilhelms-Universit\"at Bonn, Mathematisches Institut, Endenicher Allee 60,\newline 53115 Bonn, Germany\newline
	winges@math.uni-bonn.de}
}

\numberwithin{equation}{section}
\counterwithout{footnote}{section}

\newtheorem{theorem}{Theorem}[section] 
\newtheorem{prop}[theorem]{Proposition}

\newtheorem{lem}[theorem]{Lemma}
\newtheorem{kor}[theorem]{Corollary}
\theoremstyle{remark}
\theoremstyle{definition}
\newtheorem{ddd-alt}[theorem]{Definition}
\newtheorem{ex-alt}[theorem]{Example}
\newtheorem{rem-alt}[theorem]{Remark}

\newenvironment{ddd}    
{%
	\pushQED{\qed}\begin{ddd-alt}}
	{\popQED\end{ddd-alt}}

\newenvironment{ex}    
{%
	\pushQED{\qed}\begin{ex-alt}}
	{\popQED\end{ex-alt}}

\newenvironment{rem}    
{%
	\pushQED{\qed}\begin{rem-alt}}
	{\popQED\end{rem-alt}}

\crefname{lem}{Lemma}{Lemmas}
\crefname{ddd-alt}{Definition}{Definitions}

\newcommand{\cNerve}{\mathrm{N}^{\mathrm{coh}}}

\newcommand{\free}{\mathrm{free}}

 \newcommand{\Add}{\mathbf{Add}}

\newcommand{\All}{\mathbf{All}}

\DeclareMathOperator{\Res}{Res}

\newcommand{\Orb}{\mathbf{Orb}}

\newcommand{\BC}{\mathbf{BornCoarse}}

\newcommand{\bB}{{\mathbf{B}}}

\newcommand{\Groupoids}{\mathbf{Groupoids}}

\newcommand{\bF}{{\mathbf{F}}}

\newcommand{\cW}{{\mathcal{W}}}

\newcommand{\bA}{{\mathbf{A}}}

\newcommand{\const}{{\mathtt{const}}}

\newcommand{\cU}{{\mathcal{U}}}

 \newcommand{\Cat}{{\mathbf{Cat}}}

\DeclareMathOperator{\proj}{Proj}

\newcommand{\preAdd}{\mathbf{preAdd}}
 
\newcommand{\mi}{\mathrm{mi}}
\newcommand{\ma}{\mathrm{ma}}
\newcommand{\bbI}{\mathbb{I}}
\newcommand{\Lin}{\mathrm{Lin}}
\newcommand{\Free}{\mathrm{Free}}

\newcommand{\Quivers}{\mathbf{DirGraph}}
\newcommand{\Banach}{\mathbf{Ban}}

\newcommand{\fg}{\mathrm{fg}}
\newcommand{\can}{\mathrm{can}}
\newcommand{\idem}{\mathrm{idem}}
\renewcommand{\free}{\mathrm{free}}
\renewcommand{\proj}{\mathrm{proj}}
\renewcommand{\End}{\mathrm{End}}
\renewcommand{\Hom}{\mathrm{Hom}}

\begin{document}
\maketitle

\begin{abstract}
We construct combinatorial model category structures on the categories of   (marked) categories and  (marked) pre-additive categories, and we characterize  (marked) additive categories as fibrant objects in a Bousfield localization of pre-additive categories. These model category structures  are used to present the corresponding  $\infty$-categories obtained by inverting equivalences. We apply these results to explicitly  calculate   limits and colimits in these $\infty$-categories. The motivating application is a systematic construction of the equivariant coarse algebraic $K$-homology with coefficients in an additive category from its non-equivariant version.
\end{abstract}

\tableofcontents

\section{Introduction}

If $\cC$ is a category and $W$ is a set of morphisms in $\cC$, then one can consider the localization functor
\[ \ell_{\cC} \colon \cC\to\cC_{\infty}:=\cC[W^{-1}] \]
in $\infty$-categories  \cite[Def.~1.3.4.1]{HA} \cite[Def.~7.1.2]{cisin}), where we consider $\cC$ as an $\infty$-category  given by its nerve (which we will omit in the notation).   If the relative category $(\cC,W)$ extends to a simplicial model category
in which all
objects are cofibrant, then we have   an equivalence of $\infty$-categories
\[ \cC_{\infty}\simeq   \cNerve(\cC^{cf})\ ,\]
where the right-hand side is the nerve of the simplicial category of cofibrant/fibrant objects of $\cC$ \cite[Def.~1.3.4.15 \& Thm.~1.3.4.20]{HA}.   This explicit description of $\cC_{\infty}$   is sometimes very helpful in order to calculate mapping spaces in $\cC_{\infty}$ or to identify limits or colimits of diagrams in $\cC_{\infty}$.
 
In the present paper we consider the case where $\cC$ belongs to the list
\[ \{\Cat,\Cat^{+},\preAdd,\preAdd^{+}\} \]
where $\Cat^{(+)}$ is the category of small  (marked)   categories (\cref{erboi33w23f234f2f}), and $\preAdd^{(+)}$ is the category of small (marked) 
 pre-additive categories (\cref{fbioerwwef32r23r23r,rbgeoirgergergergre2r4}), and $W$ are the 
 (marking preserving)  morphisms (functors or $\Ab$-enrichment preserving functors, respectively) which admit inverses up to (marked) isomorphisms (\cref{gijoorgergerg}). 

In order to fix set-theoretic issues we choose three Grothendieck universes \begin{equation}\label{rtboihgiuf4f43f34f3f3}
\cU\subset  \cV\subset  \cW\ .
\end{equation} 
The objects of $\cC$ are categories in $\cV$ which are locally $\cU$-small, while $\cC$ itself belongs to $\cW$ and is locally $\cV$-small. We will shortly say that the objects of $\cC$ are small (as already done above), and correspondingly, that $\cC$ itself is large.  
 
Our first main theorem is:
\begin{theorem}\label{ergioergergre34}
 The pair $(\cC,W)$ extends to a combinatorial, simplicial model category structure.
\end{theorem}
We refer to \cref{vgioeoerberebg} for a more precise formulation and recall that the adjective
\emph{combinatorial} means cofibrantly generated as a model category, and locally presentable as a category.  In this model category structure all objects of $\cC$ are cofibrant.

The assertion of \cref{ergioergergre34} in the case of $\Cat$  and $\preAdd$ is well-known or folklore.
In the proof, which closely follows the standard line of arguments, we therefore put the emphasis on checking that 
all arguments work in the marked cases as well.

In order to describe the homotopy theory of (marked) additive categories,
we show the following.

\begin{prop}\label{bhergerger}
	There exists a Bousfield localization $L\preAdd^{(+)}$ of $\preAdd^{(+)}$ whose  fibrant objects  are the marked (additive) categories.
\end{prop}

We refer to  \cref{rigerogergergre} for a more precise statement.
Let  $W_{\Add^{(+)}}$ denote the weak equivalences in $L\preAdd^{(+)}$.
\cref{bhergerger} then implies that we have an equivalence of $\infty$-categories
\begin{equation}\label{g5g45ggg3ff3f}
 \Add^{(+)}_\infty := \preAdd^{(+)}[W_{\Add^{(+)}}^{-1}] \simeq \cNerve(\Add^{(+)})\ ,
\end{equation}
where $\Add^{(+)}$ denotes the category of small (marked) additive categories (see \cref{rioehgjoifgregregegergeg,reiuheriververvec}).
For example, this allows us to calculate limits in $\Add^+_\infty$,
which is one of the motivating applications of the present paper (see \cref{ex:bc}).

Since in general an $\infty$-category modeled by a combintorial model category is presentable, we get the following (see \cref{fiowefwefwfwf}).
\begin{kor}
The $\infty$-categories in the list
 \[ \{\Cat_{\infty},\Cat^{+}_{\infty},\preAdd_{\infty},\preAdd^{+}_{\infty}, \Add_{\infty},\Add^{+}_{\infty}\} \]
are presentable.
 \end{kor}
Presentability is a  very useful property  if one wants to show the existence of adjoint functors. For example the inclusion
$\cF_{\oplus} \colon \Add_{\infty}\to \preAdd_{\infty}$ preserves limits (by inspection) and therefore has a left-adjoint, the
additive completion functor
\[ L_{\oplus} \colon \preAdd_{\infty}\to \Add_{\infty} \]
(see \cref{fiowefwefwfwf}).

We demonstrate the utility of the model category structures, whose existence is asserted in \cref{ergioergergre34}, in a variety of examples.
  
  \begin{enumerate}
  \item In \cref{prop:2nerv}, we use relation \eqref{g5g45ggg3ff3f}  in order to show  an equivalence of $\infty$-categories
  \[ \Add_{\infty}\simeq  \Nerve_{2}(\Add_{(2,1)})\ ,\]
  where the right-hand side is  the $2$-categorical  nerve of the strict two-category  of small  additive 
  categories. This is used in \cite{coarsetrans} to extend $K$-theory functors from $\Add$ to $\Nerve_{2}(\Add_{(2,1)})$.  
      \item In \cref{efweoifoewfewfewf3r323r2r} we verify that the localization functor $\ell_{\cC} \colon \cC\to \cC_{\infty}$ preserves arbitrary products, where $\cC$ belongs to the list
  \[ \{ \Cat,\Cat^{+},\preAdd_{\infty},\preAdd^{+}_{\infty}, \Add_{\infty},\Add^{+}_{\infty}\}\ ,\]
  see \cref{wefiojewwefewf43t546466}.
 \item In  \cref{rgiuerhgweergergergeg} we consider additive categories of modules over rings. For example, we show in \cref{gueiurgrgerger} that
 \[ L_{\oplus}(\ell_{\preAdd}(\bR))\simeq \ell_{\Add} (\Mod^{\fg ,\free}(R) )\ ,\]
 i.e.\ that the additive completion of a ring (considered as an object $\ell_{\preAdd}(\bR)$ in $\preAdd_{\infty}$)  is equivalent to the additive category of its finitely generated and free modules (considered in $\Add_{\infty}$).
  We also discuss idempotent completions and its relation with the  additive category of finitely generated  projective modules along the same lines, see \cref{vgirejgoiergergergergregergergerg}.
  \item The main result in \cref{erbgkioergergergegreg}, see \cref{weoijoijvu9bewewfewfwef}, is an explicit formula for the object \[ \colim_{BG} \ell_{\preAdd^{(+)},BG}(\underline{\bA}) \]
  in $\preAdd^{(+)}$, where $\underline{\bA}$ is a (marked) pre-additive category with trivial  action of a group $G$ and   $\ell_{\preAdd^{(+)},BG}$ is induced from $\ell_{\preAdd^{(+)}}$.
    \item In \cref{gijeriogjeroigergregeg} we consider $\cC$ in $  \{\preAdd_{\infty},\preAdd^{+}_{\infty}, \Add_{\infty},\Add^{+}_{\infty}\}$.  In \cref{rgier9oger}, we provide an explicit formula for the object
    \[ \lim_{BG} \ell_{\cC,BG} (\bA) \ ,\]
    where $\bA$ is an object of $\cC$ with an action of $G$.  
    \end{enumerate}
 
In a parallel paper \cite{bunke} we consider model categoy structures on (marked) $*$-categories and a similar application to coarse homology theories including equivariant coarse topological $K$-homology.

\paragraph{Acknowledgements}
U.~Bunke and A.~Engel were supported by the SFB 1085 ``Higher Invariants''
funded by the Deutsche Forschungsgemeinschaft DFG. C.~Winges acknowledges support by the Max Planck Society and by Wolfgang L\"uck's ERC Advanced Grant ``KL2MG-interactions" (no.~662400). D.~Kasprowski and C.~Winges are members of the Hausdorff Center for Mathematics at the University of Bonn.

\section{Marked categories}\label{rgierogrg43rergerg4t3}

\subsection{Categories of marked categories and  marked  pre-additive  categories}
In this section we introduce  categories of marked categories, marked pre-additive categories and additive categories. We further describe various relations between these categories given by forgetful functors and their adjoints. We finally describe their enrichments in groupoids and simplicial sets.

Let $\bC$ be a category.
\begin{ddd}
A \emph{marking on $\bC$} is the choice of a wide 
{subgroupoid} $\bC^{+}$ of the underlying groupoid of $\bC$. 	
\end{ddd}

 \begin{ex}In this example, we name the two extreme cases of markings.
On the one hand, we can consider the minimal marking $\bC^{+}_{min}$ given by the identity morphisms of $\bC$. On the other hand, we have the maximal marking $\bC^{+}_{max}$ given by the underlying groupoid of $\bC$.
 \end{ex}

\begin{ddd}\label{erboi33w23f234f2f}A \emph{marked category} is  a pair $(\bC,\bC^{+})$ of a category and a marking.
A morphism between marked categories $(\bC,\bC^{+})\to (\bD,\bD^{+})$ is a functor
$\bC\to \bD$ which sends $\bC^{+}$ to $\bD^{+}$.
\end{ddd}

We let $\Cat^{+}$ denote the category of marked small categories and morphisms between marked categories. We have two functors
\begin{equation}
\label{eq_functor_Fplus}
\cF_{+} \colon \Cat^{+}\to \Cat\ , \quad  (\bC,\bC^{+})\mapsto \bC
\end{equation}
and
\[ (-)^{+} \colon \Cat^{+}\to \Groupoids\ , \quad (\bC,\bC^{+})\mapsto \bC^{+}\ .\]
The functor $\cF_{+}$ (which forgets the markings) fits into adjunctions
\[ \mi \colon \Cat\leftrightarrows \Cat^{+} \colon \cF_{+}\ , \quad \cF_{+} \colon \Cat^{+}\leftrightarrows\Cat \colon \ma\ ,\]
where the functors $\mi$ (mark identities) and $\ma$ (mark all isomorphisms) are given (on objects) by  
\[ \mi(\bC):=(\bC,\bC^{+}_{min}) \ ,\quad \ma(\bC):=(\bC,\bC^{+}_{max})\ , \]
and  their definition on morphisms as well as the unit and counit of the adjunctions are  the obvious ones.

\begin{ddd} \label{fbioerwwef32r23r23r}
A \emph{pre-additive category} is a category which is enriched over the category of abelian groups. 
  A \emph{morphism} between pre-additive categories is a functor which is compatible with the enrichment.
\end{ddd}

  We let $\preAdd$ denote the category of small pre-additive categories and functors which are compatible with the enrichment.

 The   forgetful functor
 (forgetting the enrichment)   is the right-adjoint of an adjunction
\begin{equation}
\label{eq_functor_FZ}
\Lin_{\Z} \colon \Cat\leftrightarrows\preAdd \colon \cF_{\Z}
\end{equation}
 whose left-adjoint is called the linearization functor.
 For a pre-additive category $\bA$ we call $\cF_{\Z}(\bA)$ the underlying category.
 
  \begin{rem} \label{vgeroihirovervbervevev}Let 
 $\bA$ be a pre-additive category. If $A$ and $B$ are two objects of $\bA$ such that the product $A\times B$ and the coproduct $A\sqcup B$ exist, then the canonical morphism
$ A\sqcup B\to A\times B$  induced by the maps $(\id_{A},0) \colon A\to A\times B$ and $(0,\id_{B}) \colon B\to A\times B$ is an isomorphism. In this case we call the  product or coproduct also the sum of $A$ and $B$ and use the notation $A\oplus B$.
\end{rem}

 \begin{ddd} \label{rbgeoirgergergergre2r4}
 We define the \emph{category   of marked  pre-additive categories}  $\preAdd^{+}$
 as the pull-back  (in   $1$-categories)
 \[\xymatrix{ \preAdd^{+}\ar[r]\ar[d]&\Cat^{+}\ar[d]^-{\cF_{+}}\\
  \preAdd\ar[r]^-{\cF_{\Z}}&\Cat}\]
{with the functors $\cF_{+}$ and $\cF_{\Z}$ from \eqref{eq_functor_Fplus} and \eqref{eq_functor_FZ}.}
 \end{ddd}
 Thus a marked pre-additive category is a pair $(\bA,\bA^{+})$ of a pre-additive category $\bA$ and {a wide subgroupoid} $\bA^{+}$ of the underlying groupoid of $\bA$, and a morphism of marked  pre-additive categories $(\bA,\bA^{+})\to (\bB,\bB^{+})$ is a functor $\bA\to \bB$ which is compatible with the enrichment and sends $\bA^{+}$ to $\bB^{+}$.

  We will denote the vertical arrow   forgetting the markings, i.e., taking the underlying pre-additive category, also by $\cF_{+}$. We have adjunctions
 \begin{equation}\label{f3rfkj34nfkjf3f3f3f43f}
\mi \colon \preAdd\leftrightarrows \preAdd^{+} \colon \cF_{+}\ , \quad \cF_{+} \colon \preAdd^{+}\leftrightarrows \preAdd \colon \ma\ ,
\end{equation}
 and
 \[ \Lin_{\Z} \colon \Cat^{+}\leftrightarrows  \preAdd^{+} \colon \cF_{\Z}\ .\] The unit of the last adjunction provides an inclusion of categories
   $\bC \to \cF_{\Z}(\Lin_{\Z}( \bC))$, and the subcategory of marked isomorphisms in 
   $\Lin_{\Z}(\bC )$ is exactly the image of $\bC^{+}$ under this inclusion.

\begin{rem}
 	Note that a sum of two addable  marked isomorphisms in a marked pre-additive category need not be marked. So in general the  subcategory of marked isomorphisms  of a marked pre-additive category is not  pre-additive.
\end{rem}

 From now one we will usually shorten the notation and denote marked categories just by one symbol $\bC$ instead of $(\bC,\bC^{+})$.

The categories $\Cat$, $ \Cat^{+}$ $\preAdd$ and  $\preAdd^{+}$  are enriched over themselves. For categories $\bA$ and $\bB$ we let $\Fun_{\Cat}(\bA,\bB)$ in $\Cat$ denote the category of functors from $\bA$ to $\bB$ and natural transformations.  
Assume now that $\bA$ and $\bB$ are marked. Then we can consider the
functor category $\Fun_{\Cat^{+}}(\bA,\bB)$ in $\Cat$ of functors preserving the marked subcategories and  natural transformations.

\begin{ddd} \label{gwiogefwerfwefwefewfw}
We define the \emph{marked} functor category  $\Fun_{\Cat^+}^{+}(\bA, \bB)$  in $\Cat^{+}$  
by marking those natural transformations $(u_{a})_{a\in \bA}$ of $\Fun_{\Cat^+}(\bA,\bB)$   for which $u_{a}$ is a marked isomorphism for  every $a$ in $\bA$.
\end{ddd}

Similarly, assume that $\bA$ and $\bB$ are pre-additive categories.
Then
the category of (enrichment preserving) functors $\Fun_{\preAdd}(\bA,\bB)$  and natural transformations is itself  naturally enriched in abelian groups,
and hence is an object of $\preAdd$.  If $\bA$ and $\bB$ are marked pre-additive categories, then the same applies to the category   $\Fun_{\preAdd^{+}}(\bA,\bB)$ of functors preserving the enrichment and the marked subcategories.

\begin{ddd} \label{gwiogefwerfwefwefewfwadd}
We define the \emph{marked} functor category  $\Fun^{+}_{\preAdd^{+}}(\bA, \bB)$  in $\preAdd^{+}$  
by marking those natural transformations $(u_{a})_{a\in \bA}$ of $\Fun_{\preAdd^{+}}(\bA,\bB)$   for which $u_{a}$ is marked for every $a$ in $\bA$.
\end{ddd}

\begin{rem} This is a remark about notation.
For $\cC=\Cat$ or $\cC=\preAdd$ and $\bA,\bB$ in $\cC^{+}$ we can consider the functor category
$\Fun_{\cC^{+}}(\bA,\bB)$ in $\cC$. The $+$-sign  indicates that we only consider functors which preserve marked isomorphisms. In general we have a full inclusion of categories
$\Fun_{\cC^{+}}(\bA,\bB)\subseteq \Fun_{\cC}(\cF_{+}(\bA),\cF_{+}(\bB))$.
The upper index $+$ in $\Fun^{+}_{\cC^{+}}(\bA,\bB)$  indicates that we consider
the functor category as a marked category, i.e., as an object of $\cC^{+}$. The symbol $\Fun^{+}_{\cC^{+}}(\bA,\bB)^{+}$ denotes the subcategory of marked isomorphisms. In our longer pair notation  for marked objects we thus have
\[\Fun^{+}_{\cC^{+}}(\bA,\bB)=(\Fun_{\cC^{+}}(\bA,\bB),\Fun^{+}_{\cC^{+}}(\bA,\bB)^{+})\ .\qedhere\]
\end{rem}

We now introduce enrichments of the categories over simplicial sets using the nerve functor
\[ \Nerve \colon \Cat\to \sSet\ .\]

\begin{rem} The usual enrichment of $\Cat$ over  simplicial sets is given by setting
\[ \Map^{\mathrm{standard}}_{\Cat}(\bA,\bB) :=\Nerve(\Fun_{\Cat}(\bA,\bB))\ .\]
In the present paper we will consider a different enrichment which only takes the invertible natural transformations between functors into account. 
\end{rem}

For the rest of this section $\cC$ serves as a placeholder for either $\Cat$ or $\preAdd$.

We start with marked categories  $\bA$ and $\bB$ in $\cC^{+}$.
\begin{ddd}\label{ergeiorge4tgergregergreg}
We define
\[\Map_{\cC^{+}}(\bA,\bB):=\Nerve(\Fun^{+}_{\cC^{+}}(\bA,\bB)^{+})\ .\qedhere\]
\end{ddd}
 In other words, $\Map_{\cC^{+}}(\bA,\bB)$ is the nerve of the groupoid of marked isomorphisms in $\Fun^{+}_{\cC^{+}}(\bA,\bB)$. 
 
Let now  $\bA$ and $\bB$ be  categories in $\cC$
\begin{ddd}
We define
\[\Map_{\cC}(\bA,\bB):=\Nerve(\Fun^{+}_{\cC^{+}}(\ma(\bA),\ma(\bB))^{+})\ .\qedhere\]
\end{ddd}

In other words, $\Map_{\cC}(\bA,\bB)$ is the nerve of the groupoid of  isomorphisms  in $\Fun_{\cC}(\bA,\bB)$. 

The composition of functors and natural transformations naturally induces the composition law for these mapping spaces. In this way we have turned the categories $\Cat$, $\Cat^{+}$, $\preAdd$ and $\preAdd^{+}$ into simplicially enriched categories.

\begin{rem}
Since the mapping spaces are nerves of groupoids they are Kan complexes. Therefore these simplicial categories are fibrant in Bergner's model structure on simplicial categories \cite{bergner}. \end{rem}

 \subsection{The model categories \texorpdfstring{$\preAdd^{+}$}{preAdd-plus} and \texorpdfstring{$\Cat^{+}$}{Cat-plus}}
\label{sec:marked}

 In this section we describe the model category structures on the categories $\Cat$, $\Cat^{+}$, $\preAdd$ and $\preAdd^{+}$, see  \cref{gijoorgergerg}.
The main  result is \cref{vgioeoerberebg}. 

As before, $\cC$ serves as a placeholder for either $\Cat$ or $\preAdd$.
We  first introduce the data for the   model category structure on $\cC$ or $\cC^{+}$.

\begin{ddd}\label{gijoorgergerg}\mbox{}\begin{enumerate} \item
		A morphism $f \colon \bA\to \bB$ in $\cC$ (or $\cC^{+}$) is a weak equivalence if it admits an inverse  $g \colon \bB\to \bA$ up to  isomorphisms (or marked isomorphisms).
		\item
		A morphism in $\cC $  (or $\cC^{+}$) is called a cofibration if it is injective on objects.
		\item
		A morphism in $\cC $ (or $\cC^{+}$) is called a fibration, if it has the right-lifting property for  trivial cofibrations.  \qedhere\end{enumerate}
\end{ddd}

The following is the main theorem of the present section.

\begin{theorem}\label{vgioeoerberebg}
	The simplicial category $\cC$ (or $\cC^{+}$) with the weak equivalences, cofibrations and fibrations as in \cref{gijoorgergerg}     is a simplicial 
	and combinatorial model category. \end{theorem}

\begin{proof}
We refer to \cite[Def.~1.1.3 and Def.~1.1.4]{hovey} or \cite[Def.~7.1.3]{MR1944041} for the axioms (M1)-(M5) for a model category and \cite[Def.~9.1.6]{MR1944041} for the additional axioms  {(M6) and (M7)} for  a simplicial model category.
For the Definition of cofibrant generation we refer to \cite[Def.~2.1.17]{hovey} or \cite[Def.~11.1.2]{MR1944041}.
Finally, a model category is called combinatorial if it is cofibrantly generated and locally presentable \cite{Dugger:aa},
 \cite[Def.~A.2.6.1]{htt}.

 \begin{enumerate}
\item In  \cref{wrfwweffewf} we verify completeness and cocompleteness (M1).  	 
		\item Weak equivalences have the two-out-of-three property (M2) by \cref{fwiowowfefwefwef334}.
		\item Weak equivalences, cofibrations and fibrations are closed under retracts (M3) by \cref{wfeoifjowefwefewfw}.
		\item Lifting along trivial cofibrations holds by definition. Lifting along trivial fibrations (M4) holds by \cref{fweiowefwefewffewf}.
		\item Existence of factorizations (M5) follows from  {\cref{cor:factorization1} and \cref{cor:factorization2}}. 
		\item Simplicial enrichment (M6) is shown by   \cref{efiuwehfiwefew23r23r32r},  and the pushout-product axiom (M7)   is verified in  \cref{foifjoewfefwefwef}.
		\item The category is cofibrantly generated by \cref{fiowejofwefewfewf}.
		\item It is locally presentable by \cref{ewdfoijfowefewfewfw}.\qedhere
	\end{enumerate}
\end{proof}

\begin{rem}
The case of $\Cat$ is well-known. 
In the following, in order to avoid case distinctions, we will only consider the marked case in full detail.
In fact, the functor $\ma \colon \cC\to \cC^{+}$ is the inclusion of a full simplicial subcategory and
the model category structure is inherited. We will indicate the necessary modifications
(e.g, list the generating (trivial) cofibrations or the generators of the category  in the unmarked case) in remarks at the appropriate places. 
\end{rem}

Completeness and cocompleteness in the following means  admitting limits and colimits with indexing categories in the universe $\cU$, see \eqref{rtboihgiuf4f43f34f3f3}.

\begin{prop}\label{wrfwweffewf}
	The category $\cC^{+}$ is complete and cocomplete.
\end{prop}
\begin{proof}
	We will deduce the marked case from the unmarked one and use as a known fact  that $ \cC$  is complete and cocomplete, see \cite[Prop.~5.1.7]{Borceux} for cocompleteness for $\cC=\Cat$.
	
	Let  $I$ be a   category in $\cU$  (see \eqref{rtboihgiuf4f43f34f3f3}) and $X \colon I\to \cC^{+}$ be a diagram.
	We   form the object $\colim_{I}  \cF_{+}(X)$ of $\cC$.
	We have a canonical morphism  $\cF_{+}(X)\to \underline{\colim_{I}  \cF_{+}(X)}$, where $\underline{-}$ denotes the constant $I$-object. We define the marked subcategory of $\colim_{I}  \cF_{+}(X)$ as the subcatgeory generated 
	 by the images of marked isomorphisms under the canonical functors $\cF_{+}(X(i))\to\colim_{I}  \cF_{+}(X) $ for all $i$ in $I$ and denote the resulting object of $\cC^{+}$ by $Y$.
	We claim that the resulting morphism $X\to \underline{Y}$ represents the colimit of the diagram $X$. 	
	If $Y\to T$ is a morphism in $ \cC^{+}$, then
	the  induced functor
	$\cF_{+}(X)\to \underline{\cF_{+}(Y)}\to \underline{\cF_{+}(T)}$  preserves marked isomorphisms, i.e., refines to a morphism in $(\cC^{+})^{I}$.
	Vice versa, if
	$X\to \underline{T}$ is a morphism in $ (\cC^{+})^{I}$,
	then we get an induced morphism
	$ \cF_{+}(Y)\to  \underline{\cF_{+}(T)}$.
	It preserves marked isomorphisms and therefore refines to a morphism in $\cC^{+}$.
	This shows that $\cC^{+}$ is cocomplete.	
	
	Let $X \colon I\to \cC^{+}$ again be a diagram.
	We form the object $\lim_{I}  \cF_{+}(X)$ of $\cC$.
	We have a canonical  morphism $  \underline{\lim_{I}  \cF_{+}(X)}  \to \cF_{+}(X)$.
	We mark all isomorphisms in $\lim_{I}  \cF_{+}(X)$ whose evaluations at  every $i$ in $I$ are marked isomorphisms in $X(i)$. In this way we define an object $Y$ of $\cC^{+}$.
		We claim that the resulting morphism $  \underline{Y}\to X$ represents the limit  of the diagram $X$.
	
	If $T\to Y$ is a morphism in $\cC^{+}$, then the induced
	$\underline{\cF_{+}(T)} \to \underline{\cF_{+}(Y)}  \to \cF_{+}(X)$  refines to   a morphism in $(\cC^{+})^{I}$.
	Vice versa, if
	$\underline{T}\to X$ is a morphism in $(\cC^{+})^{I}$, then we get an induced morphism
	$\underline{\cF_{+}(T)}\to   \cF_{+}(Y)$ which again refines to a morphism in $\cC^{+}$. 
	This shows that $\cC^{+}$ is complete.
\end{proof}

  We let
  \[ \cF_{\All} \colon \cC^{+}\to \Cat\]
  denote the functor which takes the underlying category, i.e., which forgets markings and enrichments (in the case of $\preAdd^{+}$).
   Recall further that we have the functor
   \[ (-)^{+} \colon \cC^{+}\to \Groupoids \]
   taking the groupoid of marked isomorphisms.

 Let  $f\colon\bA\to\bB$ be a morphism in $\cC^{+}$.\begin{lem}
	\label{lem:markedequivs}The following are equivalent. \begin{enumerate}
	\item $f$  is a weak equivalence.
	\item $\cF_{\All}(f)$ and $f^{+}$ are equivalences in $\Cat$ and $\Groupoids$, respectively.
	 \end{enumerate}
\end{lem}
\begin{proof}
If $f$ is a weak   equivalence, then by   \cref{gijoorgergerg} there exists an inverse $g$ up to marked isomorphism.
Then $\cF_{\All}(g)$ and $g^{+}$ are the required inverse equivalences of $\cF_{\All}(f)$ and $f^{+}$.

We now show the converse.	 We can choose an inverse equivalence $g^+ \colon \bB^+ \to \bA^+$ of $f^{+}$ and a natural  isomorphism $u\colon \id_{\bB^+} \xrightarrow{\cong} f^+g^+$. 	We then define a functor $g \colon \bB \to \bA$ as follows. \begin{enumerate} \item
	On objects:   For an object $B$ of $\bB$ we set $g(B) := g^+(B)$.  \item On morphisms: On the set of morphisms $\Hom_\bB(B,B')$, we define $g$ as the composition
	\[ \Hom_\bB(B,B') \xrightarrow{\cong} \Hom_\bB(fg(B),fg(B')) \xleftarrow{\cong} \Hom_\bA(g(B),g(B'))\ .\]
	 	There the first isomorphism is induced by $u$ and the second isomorphism employs the fact that $\cF_{\All}(f)$ is an equivalence. Since $u$ is given by marked isomorphisms and $ f$  induces a bijection on marked isomorphisms, this map also preserves marked isomorphisms. \end{enumerate}

Then $g$ is the required  inverse of $f$ up to marked isomorphism.  The natural transformations are $u$ and $v\colon \id_{\bA}\to gf$ determined by $f(v_A)=u_{f(A)}$. Note that both are by marked isomorphisms since $f$ is a bijection on marked isomorphisms.
\end{proof}

Note that a weak equivalence not only preserves marked isomorphisms, but also detects them.

Let $\bC$ and $\bD$ be two objects of $\cC^{+}$ and $a \colon \bC\to \bD$ be a morphism. \begin{ddd}\label{wfiojowefewfewfew} The morphism
	$a$ is called  a {\emph{marked isofibration}}, if for every object $d$ of $ \bD$, every object $c$ of $\bC$ and every marked isomorphism  $u \colon a(c)\to d$ in $\bD$  there exists a marked isomorphism $v \colon c\to c^{\prime}$ in $\bC$ such that $a(v)=u$.
\end{ddd}

\begin{ex} \label{efuweifo24frergergreg} 
The object classifier in $\Cat$ is the category $\Delta_{\Cat}^{0}$ with one object $*$ and one morphism $\id_{*}$. 
The object classifier in $\Cat^{+}$ is given by $\Delta^{0}_{\Cat^{+}}:=\mi(\Delta_{\Cat}^{0})$.
Furthermore, the object classifiers in $\preAdd$ and $\preAdd^{+}$ are given by
$\Delta_{\preAdd}^{0}:=\Lin_{\Z}(\Delta_{\Cat}^{0})$ and
$\Delta_{\preAdd^{+}}^{0}:=\Lin_{\Z}(\Delta_{\Cat^{+}}^{0})$, respectively.

The morphism classifier in $\Cat$ is the category $\Delta_{\Cat}^{1}$ with two objects $0$ and $1$, and one non-identity morphism $0\to 1$. 
The morphism classifier in $\Cat^{+}$ is given by $\Delta^{1}_{\Cat^{+}}:=\mi(\Delta_{\Cat}^{1})$.
Furthermore, the morphism classifiers in $\preAdd$ and $\preAdd^{+}$ are given by
$\Delta_{\preAdd}^{1}:=\Lin_{\Z}(\Delta_{\Cat}^{1})$ and
$\Delta_{\preAdd^{+}}^{1}:=\Lin_{\Z}(\Delta_{\Cat^{+}}^{1})$, respectively.

The invertible morphism classifier in $\Cat$ is the category $\bbI_{\Cat}$ with two objects $0$ and $1$, and   non-identity morphisms $0\to 1$ and its inverse $1\to 0$.
The invertible morphism classifier in $\Cat^{+}$ is given by $\bbI_{\Cat^{+}}:=\mi(\bbI_{\Cat})$.
Furthermore, the invertible morphism classifiers in $\preAdd$ and $\preAdd^{+}$ are given by
$\bbI_{\preAdd} :=\Lin_{\Z}(\bbI_{\Cat})$ and
$\bbI_{\preAdd^{+}} :=\Lin_{\Z}(\bbI_{\Cat^{+}})$, respectively.

Finally, the marked isomorphism classifier in $\Cat^{+}$ is given by $\bbI_{\Cat^{+}}^{+}:=\ma(\bbI_{\Cat})$,  and the one in  $\preAdd^{+}$ is given by
$\bbI^{+}_{\preAdd^{+}} :=\Lin_{\Z}(\bbI^{+}_{\Cat^{+}})$.
\end{ex}

We have the following statement about morphisms in   $\cC^{+}$.
 
	\begin{lem}\label{hgionhbaiovnioaghiphn}\mbox{}
		\begin{enumerate}
			\item Trivial fibrations are surjective on objects.
			\item Weak equivalences which are surjective on objects have the right lifting property with respect to all cofibrations.
		\end{enumerate}
		In particular, a weak equivalence is a trivial fibration if and only if it is surjective on objects.
\end{lem}
\begin{proof}
	Let $f \colon \bC \to \bD$ be a trivial fibration and let $D$ in $\bD$ be an object. Since $f$ is a weak equivalence, there exists an object $C$ in $\bC$ and an isomorphism $d \colon f(C) \xrightarrow{\cong} D$.
	Consider the commutative diagram
	\[\xymatrix{
		\Delta^0_{\cC^{+}}\ar[r]^-{C}\ar[d] & \bC\ar[d]^-{f} \\
		\bbI_{\cC^{+}}\ar[r]^-{d} & \bD
	}\]
   Since $\Delta^0_{\cC^{+}} \to \bbI_{\cC^{+}}$ is a trivial cofibration, $d$ admits a lift $c$ to $\bC$ whose codomain is a preimage of $D$.
   
   Let now $f \colon \bC \to \bD$ be a weak equivalence which is surjective on objects. Consider a commutative diagram
   \[\xymatrix{
   	\bA\ar[r]^-{\alpha}\ar[d]_-{i} & \bC\ar[d]^-{f} \\
   	\bB\ar[r]^-{\beta} & \bD
   }\]
   in which $i$ is a cofibration.
   
   We first define the lift $\gamma$ of $\beta$ on objects.
   If $B$ in $\bB$ lies in the image of $i$, there exists a unique object $A$ in $\bA$ with $i(A) = B$, and we set $\gamma(B) = \alpha(A)$. Otherwise, pick any $C$ in $\bC$ such that $f(C) = \beta(B)$ and set $\gamma(B) = C$.   
   For a morphism $b$ in $\bB$,
   define $\gamma(b)$ as the unique preimage of $\beta(b)$ under $f$.
   Then $f \circ \gamma = \beta$ holds by definition, and $\gamma \circ i = \alpha$ also follows easily from the fact that $f$ is faithful. 
\end{proof}

\begin{lem}\label{fweiojweoiffewfwefwef}
	A morphism in $\cC^{+}$ is a marked isofibration if and only if it has the right lifting property with respect to the morphism 
	\[ \Delta^{0}_{\cC^{+}} \xrightarrow{0} \bbI^{+}_{\cC^{+}} \] classifing the object $0$ of $ \bbI^{+}_{\cC^{+}}$.
\end{lem}
\begin{proof}
	In view of the universal properties of $\Delta^{0}_{\cC^{+}}$ and $\bbI_{\cC^{+}}^{+}$ this is just a reformulation of \cref{wfiojowefewfewfew}.
\end{proof}

Since $\Delta^{0}_{\cC^{+}} \xrightarrow{0} \bbI^{+}_{\cC^{+}}$ is a trivial cofibration  we conclude that
fibrations are marked isofibrations. 

\begin{prop}\label{fiuehfieufwfewfwef}
	The marked isofibrations in $\cC^{+}$ have the  right lifting property with respect to trivial cofibrations.
\end{prop}
\begin{proof}
	We consider a diagram    
	\[\xymatrix{\bA\ar[r]^{\alpha}\ar[d]^{i}&\bC\ar[d]^{f}\\\bB\ar@{.>}[ur]^{\ell}\ar[r]^{\beta}&\bD} \]
	in $\cC^{+}$,
	where $f$ is a marked isofibration and $i$ is a trivial cofibration.
	We can find now a morphism $j \colon \bB\to \bA$ such that $j\circ i=\id_{\bA}$ and such that there is a marked isomorphism $u \colon i\circ j\to \id_{\bB}$  which in addition satisfies $u\circ i=\id_{i}$.

	 {On objects we define $\ell$ as follows: 
		For every object $B$ of $\bB$ we get a marked isomorphism \[\beta(u_{B}) \colon f(\alpha(j(B))=\beta(i(j(B)))\to  \beta(B)\ .\] Using that $f$ is a marked isofibration we choose a marked isomorphism $v_B \colon \alpha(j(B))\to C$ such that $f(v_B)=\beta(u_{B})$. If $B$ is in the image of $i$, we can and will choose $v_B$ to be the identity. We then set $\ell(B):=C$. This makes both triangles commute.}
	
	We now define the lift $\ell$ on a morphism $\phi \colon B\to B^{\prime}$ by 
		\[\ell(\phi):=v_{B'}\circ \alpha(j(\phi))\circ v_B^{-1}\ .\]
	
	One can check that then both triangles commutes and that this really defines a functor.
	One further checks that $\ell$ is a morphism of marked categories (and preserves the enrichment in the case of pre-additive categories). Here we use that $i$ detects marked isomorphisms. 
\end{proof}

\begin{kor}\label{wfeiweiofewfewfewf}
 {The notions of marked isofibration and fibration in $\cC^{+}$ coincide.}
\end{kor}

\begin{rem}\label{fiowefwefewfewf}
	We note that all objects in $\cC^{+}$ are fibrant and cofibrant.
Consequently, the model category  $\cC^{+}$ is proper by \cite[Cor.~13.1.3]{MR1944041}
\end{rem}

\begin{prop}\label{fweiowefwefewffewf}
	The cofibrations in $\cC^{+}$ have the left-lifting property with respect to trivial fibrations.
\end{prop}
\begin{proof}
 We consider a diagram    
		\[\xymatrix{\bA\ar[r]^{\alpha}\ar[d]^{i}&\bC\ar[d]^{f}\\\bB\ar@{.>}[ur]^{\ell}\ar[r]^{\beta}&\bD} \] 
		in $\cC^{+}$,
		where $f$ is a trivial fibration and $i$ is a cofibration.
	
	Since the map $i$ is injective on objects and the morphism $f$ is surjective on objects by \cref{hgionhbaiovnioaghiphn},
	we can find a lift $\ell$ on the level of objects. Let now $B,B^{\prime}$ be objects in $\bB$. 
	Since $f$ is fully faithful  we have a bijection 
	\[\xymatrix{\Hom_{\bC}(\ell(B),\ell(B^{\prime})) \ar[r]^f_\cong&\Hom_{\bD}(\beta(B),\beta(B^{\prime}) )}  \ .\]
	We can therefore define $\ell$ on   $\Hom_{\bB}(B,B^{\prime})$ by
	\[ \Hom_{\bB}(B,B^{\prime})\xrightarrow{\beta}  \Hom_{\bD}(\beta(B),\beta(B^{\prime}))\cong \Hom_{\bC}(\ell(B),\ell(B^{\prime}))\ .\]
	Since $f$ detects marked isomorphisms, $\ell$ preserves them.
	The lower triangle commutes by construction. One can furthermore check that the upper triangle commutes. Finally one checks that this really defines a functor.  
\end{proof}

\begin{lem}\label{fwiowowfefwefwef334}
	The weak equivalences in $\cC^{+}$ satisfy the two-out-of-three axiom. 
\end{lem}
\begin{proof}
	It is clear that the composition of weak equivalences is a weak equivalence. 
	Assume that $f \colon \bA\to \bB$ and $g \colon \bB\to \bC$ are morphisms such that $f$ and $g\circ f$ are weak equivalences. Then we must show that $g$ is a weak equivalence. Let
	$m \colon \bB\to \bA$ and $n \colon \bC\to \bA$ be inverse functors and let
	$u \colon m\circ f\to \id_{\bA}$, $v \colon f\circ m\to \id_{\bB}$,
	$x \colon n\circ g\circ f\to \id_{\bA}$ and $y \colon g\circ f\circ n\to \id_{\bC}$ be the corresponding marked
	isomorphisms.
	Then we consider the functor
	$h:=f\circ n \colon \bC\to \bB$. We have marked isomorphisms
	\[ h\circ g=f\circ n\circ g\xrightarrow{ v^{-1} } f\circ n\circ g\circ f\circ m\xrightarrow{x}  f\circ m\xrightarrow{v} \id_{\bB}\ .\]
	and
	 \[g\circ h=g\circ f\circ n\xrightarrow{y} \id_{\bC}\ .\]
	 If $g$ and $g \circ f$ are weak equivalences, the argument is analogous.
\end{proof}

\begin{prop}\label{wfeoifjowefwefewfw}
	The cofibrations, fibrations, and weak equivalences in $\cC^{+}$ are closed under retracts.
\end{prop}
\begin{proof}
	Since fibrations are characterized by a right lifting property they are closed under retracts.
	Cofibrations are closed under retracts since a retract diagram of marked categories induces a retract diagram on the level of sets of objects, and injectivity of maps between  sets is closed under retracts.
	It remains to consider weak equivalences. We consider a diagram
	\[\xymatrix{\bA\ar[r]^{i}\ar[d]^{f}&\bA^{\prime}\ar[r]^{p}\ar[d]^{f^{\prime}}&\bA\ar[d]^{f}\\\bB\ar[r]^{j}&\bB^{\prime}\ar[r]^{q}&\bB}\]
	in $\cC^{+}$ with $p\circ i=\id_{\bA}$ and $q\circ j=\id_{\bB}$, and where $f^{\prime}$ is a weak equivalence. Let $g^{\prime} \colon \bB^{\prime}\to \bA^{\prime}$ be an inverse of $f^{\prime}$ up to marked isomorphism.
	Then $p\circ g^{\prime}\circ j \colon \bB\to \bA$ is an inverse of $f$ up to marked isomorphism.
\end{proof}

We have finished the verification of the basic model category axioms except the existence of factorizations. This follows from considerations about the simplicial structure which we do now.

 {We define a functor \begin{equation}\label{rthoir3terhtrhrth}
Q\colon\Groupoids\to \cC^{(+)}
\end{equation}
as follows. Let $i\colon\Groupoids\to \Cat$ be the inclusion.
\begin{enumerate}
	\item In the case $\Cat$, we define $Q:=i$.
	\item In the case $\Cat^+$, we define $Q:=\ma\circ i$.
	\item In the case $\preAdd$, we define $Q:=\Lin_\Z\circ i$.
	\item In the case $\preAdd^+$, we define $Q:=\Lin_\Z\circ\ma\circ i$.
\end{enumerate}
For (marked) preadditive categories we need a further symmetric monoidal product structure $\otimes$ (which differs from the cartesian structure) on $\preAdd^{+}$ given as follows:
\begin{enumerate}
	\item (objects) The objects of $\bA\otimes \bB$ are pairs $(A,B)$ of objects  $A$ in $\bA$ and $B$ in $\bB$.
	\item (mophisms) The abelian group of morphisms between $(A,B)$ and $(A^{\prime},B^{\prime})$ is given by
	\[ \Hom_{\bA\otimes \bB}((A,B),(A^{\prime},B^{\prime})):=\Hom_{\bA}(A,A^{\prime})\otimes \Hom_{\bB}(B,B^{\prime})\ .\]
	The composition is defined in the obvious way.
	\item(marking) We mark tensor products of marked isomorphisms.
\end{enumerate}
We refrain from writing out the remaining data (unit, unit- and associativity constraints) explicitly.}

 In order to define a tensor structure of $\cC^{(+)}$ over simplicial sets, we start with a tensor structure over groupoids.
\begin{ddd}\label{defn_functor_sharp}
	In the case $\Cat^{(+)}$ we define the functor   \[-\sharp -\colon \Cat^{(+)} \times \Groupoids \to \Cat^{(+)} \ ,\quad (\bA,G)\mapsto \bA\sharp G:=\bA\times  Q(G).\] 
	In the case $\preAdd^{(+)}$ we define the functor   \[-\sharp -\colon \preAdd^{(+)} \times \Groupoids \to \preAdd^{(+)} \ ,\quad (\bA,G)\mapsto \bA\sharp G:=\bA\otimes  Q(G).\qedhere\] 
\end{ddd}

Let $\bB$ be in $\cC^{+}$. In the following lemma, we will write $\otimes$ for the product in $\Cat^+$, to avoid distinguishing between $\Cat^+$ and $\preAdd^+$.

 \begin{lem}\label{gioergegergreg}
	We have an adjunction
	\[ -\otimes \bB \colon \cC^{+}\leftrightarrows \cC^{+} \colon \Fun^{+}_{\cC^{+}} (\bB  ,-)\ ,\]
	where we view $\cC^+$ as enriched over $\cC^{+}$.
\end{lem}
\begin{proof}  We provide an explicit description of the unit  and the counit of the adjunction.
For $\bA$ in $\cC^+$ they are given by morphisms
		\[ \eta_{\bA} \colon \bA \to \Fun_{\cC^{+}}^+(\bB,\bA \otimes \bB) \quad \text{and} \quad \epsilon_{\bA} \colon \Fun_{\cC^{+}}^+(\bB,\bA) \otimes \bB \to \bA \]
		defined as follows: \begin{enumerate} \item 
		The morphism $\eta_{\bA}$ takes an object $A$ in $\bA$ to the functor sending an object $B$ in $\bB$ to $(A,B)$ and a morphism $b$ in $\bB$ to $(\id_A,b)$.
		 A morphism $a \colon A \to A'$ is sent by $\eta_{\bA}$ to the natural transformation $\{ (a,\id_B) \colon (A,B) \to (A',B) \}_{B \in B}$.
		 \item The morphism $\epsilon_A$ is induced by evaluation of functors.
	\end{enumerate}
	One checks that $\eta$ and $\epsilon$ are natural transformations. One furthermore 
	 checks the triangle identities by explicit calculations. 
	 \end{proof}

  Recall that	for $\bA$ and $\bB$ in $\cC^{+}$ the category  $\Fun^{+}_{\cC^{+}}(\bA,\bB)^{+}$ is a groupoid.
Let $G$ be a groupoid. From \cref{gioergegergreg}  we get natural  isomorphisms
\begin{equation} \label{egiherigergregergr1}  	 
	 \Fun_{\Groupoids}(G, \Fun^{+}_{\cC^{+}}(\bA,\bB)^{+})\cong \Fun_{\cC^{+}}^{+}(\bA\sharp G,\bB)^{{+}}\cong  \Fun^{+}_{\cC^{+}}(\bA,\Fun^{+}_{\cC^{+}}(Q(G),\bB))^{{+}}
	 \end{equation}
 	 
In order to define the tensor structure  of   $\cC^{+}$ with simplicial sets 
   we consider the  fundamental groupoid functor. 
   \begin{ddd}\label{fewiohfwiof2323r23r}  The  \emph{fundamental groupoid} functor $\Pi$  is defined as the left-adjoint  of the adjunction
\[\Pi \colon \sSet\leftrightarrows \Groupoids \colon \Nerve\ ,\]
where $\Nerve$ takes the nerve of a groupoid.\end{ddd}
Explicitly, the fundamental groupoid $\Pi(K)$ of a simplicial set  $K$ is the groupoid freely generated by the path  category $P(K)$ of $K$. The category $P(K)$ in turn is given as follows:
\begin{enumerate}
\item The objects of $P(K)$ are the $0$-simplices.
\item The morphisms of $P(K)$ are generated by the $1$-simplices of $K$ subject to the relation
$g\circ f\sim h$ if there exists a $2$-simplex $\sigma$ in $K$ with
$d_{2} \sigma=f$, $d_{0}\sigma=g$ and $d_{1}\sigma=h$.
\end{enumerate}

 Using the tensor and cotensor structure with groupoids we define the corresponding structures with simplicial sets by pre-composition with the {fundamental}-groupoid functor. Recall the definition   \eqref{rthoir3terhtrhrth} of $Q$.
 
\begin{ddd}\label{grighergregregeree}We define \emph{tensor and cotensor structures} on $\cC^{+}$ with simplicial sets   by
	\begin{equation*}
\cC^{+}\times \sSet\to \cC^{+}\ , \quad (\bA,K)\mapsto \bA\sharp \Pi(K)\ .
\end{equation*}
\begin{equation*}
\sSet^{op}\times \cC^{+}  \to \cC^{+}\ , \quad (K,\bB )\mapsto  \Fun_{\cC^{+}}^{+}(Q(\Pi(K)),\bB)\ .\qedhere
\end{equation*}

In order to simplify notation, we will usually write $\bA\sharp K$ instead of $ \bA\sharp \Pi(K)$
and $\bB^{K}$ instead of $\Fun_{\cC^{+}}^{+}(Q(\Pi(K)),\bB)$. \end{ddd}

  \cref{gioergegergreg}  has the following corollary obtained by applying the nerve functor and using  \cref{ergeiorge4tgergregergreg} of the simiplicial mapping sets in $\cC^{+}$.
 \begin{kor}\label{efiuwehfiwefew23r23r32r}
 For $K$ in $\sSet$ and $\bA$, $\bB$ in $\cC^{+}$
	 we have natural isomorphisms of simplicial sets
	 \[ \Map_{\sSet}(K,\Map_{\cC^{+}}(\bA,\bB))\cong \Map_{\cC^{+}}(\bA\sharp K,\bB)\cong  \Map_{\cC^{+}}(\bA, \bB^{K})\ .\]
	 \end{kor}

  We consider a commutative square
\begin{equation}\label{2oihvwvwewv}
\xymatrix{
			\bA\ar[r]^{i}\ar[d]_{f} & \bB\ar[d]^{g} \\
			\bC\ar[r]^{j} & \bD
		}
\end{equation}
		in $\cC^{+}$.

\begin{lem}\label{lem:push.cofib.pull.fib}
	\ 
	\begin{enumerate}
		\item \label{giorgergergerge99} If \eqref{2oihvwvwewv}
		is a pushout and $i$ is a trivial cofibration, then $j$ is a trivial cofibration.
		\item\label{giorgergergerge999} If \eqref{2oihvwvwewv}
			is a pullback and $g$ is a trivial fibration, then $f$ is a trivial fibration.
	\end{enumerate}
\end{lem}
 
\begin{proof}
We show Assertion \ref{giorgergergerge99}.  
	 	Because $i$ is a trivial cofibration, there exists a morphism $i' \colon \bB \to \bA$ such that $i' \circ i = \id_\bA$ and a marked isomorphism $u \colon i \circ i' \to \id_\bB$ satisfying $u \circ i = \id_i$.
	By the universal property of the push-out,  the morphism $f \circ i' \colon \bB \to \bC$ induces a morphism $j' \colon \bD \to \bC$ such that $j' \circ j = \id_\bC$.
	In particular, $j$ is a cofibration and it remains to show that it is a weak equivalence.
	Moreover, $g \circ u$ provides a marked isomorphism $j \circ f \circ i' = g \circ i \circ i' \to g$.
	
	The functor $-\sharp \bbI_{\Cat} \colon \cC^{+}\to \cC^{+}$  (see \cref{efuweifo24frergergreg} for $\bbI_{\Cat}$ in $\Groupoids$)
	  is a left-adjoint by \cref{gioergegergreg}. Therefore  it preserves pushouts.  
	  Using the first isomorphism in   \eqref{egiherigergregergr1} and the fact that $\bbI_{\Cat}$ is the morphism classifier in $\Groupoids$, 
	we consider  the natural transformation  $g \circ u$ as a functor $\bB \sharp \bbI_{\Cat} \to \bD$. Together with the functor $\bC\sharp \bbI_{\Cat}\to \bD$ corresponding  to the identity natural transformation of $j$,    by the universal property of  the push-out diagram $ \eqref{2oihvwvwewv} \sharp \bbI_{\Cat}$ we obtain an induced functor $\bD \sharp \bbI_{\Cat} \to \bD$ which provides, by a converse application of  the first isomorphism in   \eqref{egiherigergregergr1},
	a marked isomorphism $j \circ j' \to \id_\bD$. This proves that $j$ is a weak equivalence.
	
	The proof of Assertion \ref{giorgergergerge999} can be obtained by dualizing the proof above.\end{proof}

The following proposition verifies the pushout-product axiom (M7).

\begin{prop}\label{foifjoewfefwefwef}Let $a \colon \bA \to \bB$ be a cofibration in $\cC^{+}$ and
	$i \colon X\to Y$ be a cofibration in $\sSet$.
	Then  \begin{equation}\label{dquihiduqwdqwd}
	(\bA\sharp  Y )\sqcup_{\bA\sharp X } (\bB\sharp  X )\to (\bB\sharp  Y )
	\end{equation}
	is a cofibration. Moreover, if $i$ or $a$ is in addition a weak equivalence, then
	\eqref{dquihiduqwdqwd} is a weak equivalence.
\end{prop}

In the proof of this proposition we use the following two lemmas.

\begin{lem} \label{reiofweiofweewf}
	For $\bA$ in $ \cC$ the functor
	\[ \bA\sharp  - \colon\sSet\to \cC^{+} \]
	preserves (trivial) cofibrations.
\end{lem}
\begin{proof}
	If $i \colon X\to Y$ is a cofibration, then $\Pi(i)$ is injective on objects. This implies that
	$\bA\sharp  i $ is injective on objects.
	
	Assume now that $i $ is  in addition a weak equivalence. Then
	$\Pi(i)$ is an equivalence of groupoids. Let $j \colon \Pi(Y)\to \Pi(X)$ be an inverse equivalence and
	$u \colon j\circ \Pi(i)\to \id_{\Pi(X)}$ and $v \colon \Pi(i)\circ j\to \id_{\Pi(Y)}$ be the corresponding  isomorphisms.
	Then we get a marked isomorphism  
	\[ \bA\sharp u \colon (\bA\sharp j)\circ (\bA\sharp  i )\to \id_{\bA\sharp  X }\]
	by $(\bA\sharp u)_{(a,x)}:=(\id_{a},u_{x})$. 
	Similarly, we have a marked isomorphism
	\[ p \bA\sharp v \colon (\bA\sharp   i )\circ (\bA\sharp j) \to \id_{\bA\sharp  Y }\]
	given by $(\bA\sharp v)_{(a,x)}:=(\id_{a},v_{x})$.
\end{proof}

\begin{lem}\label{efwiuhfuihfewihfweiufhewfewfewfewfwefwf}
	For a  simplicial set  $K$, the functor
	\[ -\sharp K \colon \cC^{+}\to \cC^{+} \]
	preserves (trivial) cofibrations.
\end{lem}
\begin{proof}
	If $a \colon \bA\to \bB$ is a cofibration, then it is injective on objects. Then
	$a\sharp K$ is injective on objects and hence a cofibration.
	If $a$ is in addition a marked equivalence, then $a\sharp K$ is a marked  equivalence, too. The argument is similar to 
	the corresponding part of the argument  in the proof of \cref{reiofweiofweewf}.
\end{proof}

\begin{proof}[Proof of \cref{foifjoewfefwefwef}]
	Consider the diagram
	\[\xymatrix{\bA\sharp X \ar[rr]_{\bA\sharp i}\ar[dd]_{a\sharp  X }&&\bA\sharp  Y \ar[dd]^{a\sharp Y }\ar[dl]^-{b}\\&(\bA\sharp  Y )\sqcup_{\bA\sharp  X } (\bB\sharp  X )\ar[dr]^-{?}&\\\bB\sharp   X   \ar[rr]^{\bB\sharp i}\ar[ur]&&\bB\sharp  Y }\]
	The set of objects of the push-out is equal to  the push-out of the object sets. Hence it is easy to check that $?$ is injective on objects and thus a cofibration.
	
	Assume that $a$ is a weak equivalence.
	By \cref{efwiuhfuihfewihfweiufhewfewfewfewfwefwf} the maps $a\sharp X $ and $a\sharp  Y $ are trivial cofibrations.
	 {Since $b$ is a pushout of a trivial cofibration, it is a trivial cofibration by \cref{lem:push.cofib.pull.fib}.}
	It follows from the two-out-of-three property, see \cref{fwiowowfefwefwef334}, that	the morphism $?$ is a weak equivalence.
	
	The case that $i$ is a weak equivalence is similar using \cref{reiofweiofweewf} instead of \cref{efwiuhfuihfewihfweiufhewfewfewfewfwefwf}.
\end{proof}

\begin{lem}\label{cor:factorization1}
	 Every morphism in $\cC^+$ can be factored into a cofibration followed by a trivial fibration.
\end{lem}
\begin{proof}
 Let $a \colon \bA \to \bB$ be a morphism in $\cC^+$.
 Denote by $i_1 \colon \bA \cong \bA \sharp \Delta^0 \to \bA \sharp \partial\Delta^1$ the morphism induced by the map classifying the vertex $1$, and let $j \colon \bA \sharp \partial\Delta^1 \to \bA \sharp \Delta^1$ be the morphism induced by the inclusion $\partial\Delta^1 \to \Delta^1$. Consider the diagram
 \[\xymatrix{
  \bA\ar[r]^-{i_0}\ar[d]_-{a} & \bA \sharp \partial\Delta^1\ar[r]^-{j}\ar[d] & \bA \sharp \Delta^1\ar[d] \\
  \bB\ar[r]^-{e_\bB} & \bA \sqcup \bB\ar[r]^-{b} & Z(a)
 }\]
 in which $e_\bB$ is the canonical morphism, and in which the right square is defined to be a push-out.
 Since $\bA \sharp \partial\Delta^1 \cong \bA \sqcup \bA$, it is easy to see that the left square is also a push-out.
 Hence, the outer square is also a push-out.
 By the universal property of the push-out, the composed morphism
 \[ \bA\sharp  \Delta^{1} \xrightarrow{a\sharp\Delta^{1}}\bB\sharp \Delta^{1}  \xrightarrow{\pr_{\bB}} \bB \]
 and the identity on $\bB$ induce a morphism $q \colon Z(a) \to \bB$ such that $q \circ b = \id_\bB$.
 In particular, $q$ is surjective on objects.
 Moreover, $b \circ e_\bB$ is a trivial cofibration by \cref{reiofweiofweewf} and \cref{lem:push.cofib.pull.fib}.\ref{giorgergergerge99}.
 The two-out-of-three property (\cref{fwiowowfefwefwef334}) implies that $q$ is a weak equivalence, and hence a trivial fibration by \cref{hgionhbaiovnioaghiphn}.
 
 Since the structure morphism $e_\bA \colon \bA \to \bA \sqcup \bB$ is a cofibration, the morphism $a' \colon \bA \to Z(a)$ is also a cofibration.
 Regarding $Z(a)$ as the push-out of the right square, it follows from the universal property that $q \circ (b \circ e_\bA) = a$,
 and thus provides the required factorization. 
\end{proof}

Let $\bA$ be an object of $ \cC^{+}$. Recall the notation $\bA^{K}$ for a simplicial set $K$ from \cref{grighergregregeree}.

\begin{lem}\label{lem:power.fib}
{The  functor
		\[  \bA^{(-)} \colon \sSet^{op} \to \cC^{+} \]
		sends (trivial) cofibrations to (trivial) fibrations.}
\end{lem}
\begin{proof}
This follows from  \cref{gioergegergreg}, \cref{efwiuhfuihfewihfweiufhewfewfewfewfwefwf} and \cref{reiofweiofweewf} by explicitly checking lifting properties.
\end{proof}

\begin{lem}\label{cor:factorization2}
Every morphism in $\cC^+$ can be factored into a trivial cofibration followed by a fibration.
\end{lem}
\begin{proof}
Let $a \colon \bA \to \bB$ be a morphism in $\cC^+$.
Denote by $(\ev_0,\ev_1) \colon \bB^{\Delta^1} \to \bB^{\partial\Delta^1} \cong \bB \times \bB$ the morphism induced by the canonical inclusion $\partial\Delta^1 \to \Delta^1$.
Let $p_1 \colon \bB^{\partial\Delta^1} \cong \bB \times \bB \to \bB$ denote the projection on the second factor (which corresponds to the vertex $1$), and let $p_\bA \colon \bA \times \bB \to \bA$ be the projection.
Consider the diagram
\[\xymatrix@C=4em{
 P(a)\ar[r]^-{q}\ar[d] & \bA \times \bB\ar[r]^-{p_\bA}\ar[d]^-{a \times \id_\bB} & \bA\ar[d]^-{a} \\
 \bB^{\Delta^1}\ar[r]^-{(\ev_0,\ev_1)} & \bB \times \bB\ar[r]^-{p_1} & \bB
}\]
in which the left square is defined to be a pull-back.
Since the right square is also a pull-back, the outer square is a pull-back, too.
By the universal property of the pull-back, the composed morphism
\[ \bA \xrightarrow{a} \bB \xrightarrow{\const} \bB^{\Delta^1} \]
and the identity on $\bA$ induce a morphism $i \colon \bA \to P(a)$ such that $p_\bA \circ q \circ i = \id_\bA$.
In particular, $i$ is a cofibration.
Since $\ev_1 = p_1 \circ (\ev_0,\ev_1)$ is a trivial fibration by \cref{lem:power.fib}, it follows from \cref{lem:push.cofib.pull.fib}.\ref{giorgergergerge999} that $p_\bA \circ q$ is a trivial fibration.
The two-out-of-three property (\cref{fwiowowfefwefwef334}) implies that $i$ is a weak equivalence, and thus a trivial cofibration.

Note that $q$ is a fibration since it is the pullback of a fibration (use again \cref{lem:power.fib} and \cref{lem:push.cofib.pull.fib}.\ref{giorgergergerge999}).
Since the structure morphism $p_\bB \colon \bA \times \bB \to \bB$ is a fibration, the morphism $p_\bB \circ q \colon P(a) \to \bB$ is also a fibration.
Regarding $P(a)$ as the pull-back of the left square, it follows from the universal property that $(p_\bB \circ q) \circ i = a$,
and thus provides the required factorization. 
  \end{proof}

We thus have finished the verification of the model category axioms (M1) to (M7).  

\begin{rem}
By considering the full embedding $\ma \colon \cC\to \cC^{+}$, we obtain a verification of the axioms in the unmarked case.
\end{rem}

We next describe the generating cofibrations and the generating trivial cofibrations.

Recall that by \cref{fweiojweoiffewfwefwef} and \cref{wfeiweiofewfewfewf} we can take \[J:=\{\Delta_{\cC^{+}}^{0}\to \bbI^{+}_{\cC^{+}} \}\]
as the generating trivial cofibrations for $\cC^{+}$.

\begin{rem}
The set of generating trivial cofibrations for $\cC$ is given by 
\[J:=\{\Delta^{0}_{\cC}\to \bbI_{\cC} \}\ .\qedhere\]
\end{rem}

 We  furthermore define
\[ I:=J\cup \{U,V,V^{+},W,W^{+}\} \]
where $U,V,V^{+},W,W^{+}$ are cofibrations defined as follows (see \cref{efuweifo24frergergreg}):
\begin{enumerate}
	\item $U \colon \emptyset\to \Delta_{\cC^{+}}^{0}$.
	\item We let $V \colon \Delta_{\cC^{+}}^{0}\sqcup \Delta_{\cC^{+}}^{0}\to \Delta^{1}_{\cC^{+}}$ classify the pair of objects $(0,1)$ of   $ \Delta^{1}_{\cC^{+}}$.  
	\item 
		We let $V^{+} \colon \Delta_{\cC^{+}}^{0}\sqcup \Delta_{\cC^{+}}^{0}\to \bbI^{+}_{\cC^{+}}$ classify the pair of objects $(0,1)$ of   $\bbI^{+}_{\cC^{+}}$. 
		   \item We define $P$ as the push-out
	\[\xymatrix{\Delta_{\cC^{+}}^{0}\sqcup \Delta_{\cC^{+}}^{0}\ar[r]^-V\ar[d]^{V}& \Delta^{1}_{\cC^{+}}\ar[d]\\ \Delta^{1}_{\cC^{+}}\ar[r]&P}\]
	and let $W \colon P\to  \Delta^{1}_{\cC^{+}}$ be the obvious map induced by $\id_{ \Delta^{1}_{\cC^{+}}}$.   \item We define $P^{+}$ as the push-out
	\[\xymatrix{\Delta_{\cC^{+}}^{0}\sqcup \Delta_{\cC^{+}}^{0}\ar[r]^-{V^{+}}\ar[d]^{V^{+}}&\bbI^{+}_{\cC^{+}}\ar[d]\\\bbI^{+}_{\cC^{+}}\ar[r]&P^{+}}\]
	and let $W^{+} \colon P^{+}\to \bbI^{+}_{\cC^{+}}$ be the obvious map induced by $\id_{\bbI^{+}_{\cC^{+}}}$.
\end{enumerate}

\begin{lem}
	The trivial  fibrations in $\cC^{+}$ are exactly the morphisms which have the right-lifting property with respect to $I$.
	 \end{lem}
\begin{proof}
A trivial fibration is a weak equivalence which is in addition surjective on objects by \cref{hgionhbaiovnioaghiphn}.  

We first observe that lifting with respect to $U$ exactly corresponds to the surjectivity on objects.

We now use the characterization of weak equivalences given in \cref{lem:markedequivs}.
 Lifting with respect to $V$ and $W$ corresponds to surjectivity and injectivity on morphisms, and
	lifting with respect to $V^{+}$ and $W^{+}$ corresponds to surjectivity and injectivity on marked isomorphisms. 
	  \end{proof}

\begin{lem}
	The objects $\emptyset$, $\Delta_{\cC^{+}}^{0}$, $\Delta^{1}_{\cC^{+}}$, $\bbI_{\cC}^{+}$, 
	$P$  and $P^{+}$ are   compact. \end{lem}
\begin{proof}
If $\cC=\Cat$,  then they are finite categories.
If $\cC=\preAdd$, then they have finitely many objects and finitely generated abelian morphism groups. This implies the assertion.
	 \end{proof}

\begin{rem}
In the unmarked case,
we can take the set of generating cofibrations
\[ I:=J\cup \{U,V,W\} \]
with the following definitions:
\begin{enumerate}
	\item $U \colon \emptyset\to \Delta_{\cC}^{0}$.
	\item We let $V \colon \Delta_{\cC}^{0}\sqcup \Delta_{\cC }^{0}\to \Delta^{1}_{\cC}$ classify the pair of objects $(0,1)$ of   $ \Delta^{1}_{\cC}$.  
	 		   \item We define $P$ as the push-out
	\[\xymatrix{\Delta_{\cC}^{0}\sqcup \Delta_{\cC}^{0}\ar[r]^-V\ar[d]^{V}& \Delta^{1}_{\cC}\ar[d]\\ \Delta^{1}_{\cC}\ar[r]&P}\]
	and let $W \colon P\to  \Delta^{1}_{\cC}$ be the obvious map induced by $\id_{ \Delta^{1}_{\cC}}$.   
\end{enumerate}

The objects  $\emptyset$, $\Delta_{\cC}^{0}$, $\Delta^{1}_{\cC}$, $\bbI_{\cC}$ 
	 and $P$ are    compact.
\end{rem}

\begin{kor}\label{fiowejofwefewfewf}
	The model category $\cC^{+}$ is cofibrantly generated by finite sets of generating cofibrations and trivial cofibrations between  {compact} objects. 
\end{kor}

\begin{prop}\label{ewdfoijfowefewfewfw}
	The  category $\cC^{+}$ is locally presentable.
\end{prop}
\begin{proof}
	Since we have already shown that $\cC^{+}$ is cocomplete, by \cite[Thm  1.20]{AR} it suffices to show that $\cC^{+}$ has a strong generator consisting of  compact objects. For this it suffices to show that there exists a set of compact objects such that every other object of $\cC$ is isomorphic to a colimit of a diagram with values in this set, see  \cite[Lemma 11.4]{bunke}. We will call such a set strongly generating.	
	 	
	We will first show that $\Cat^{+}$ is strongly generated by a finite set of compact objects.
	 We consider the category $\Quivers^{+}$ of marked directed graphs. It consists of directed graphs  with     distinguished subsets of edges called marked edges. Morphisms in $\Quivers^{+}$
	must preserve marked edges.
	The category $\Quivers^{+}$   is locally presentable by \cite[Thm  1.20]{AR}. Indeed, it is cocomplete and strongly  generated by the objects in the list 
	\[\{*\ , \bullet \to \bullet, \bullet\xrightarrow{+} \bullet\}\ .\]  We have a forgetful functor from $\Cat^{+}$ to marked  directed graphs which fits into an adjunction
	\[\Free_{\Cat^{+}}\colon\Quivers^{+}\leftrightarrows \Cat^{+}\colon\cF_{\circ}\ .\] 
	The left adjoint takes the free category on the marked directed graph and
	localizes at the marked isomorphisms.
	The counit of the adjunction  provides  
	a canonical morphism
	\[v_{\bA} \colon \bF(\bA):=\Free_{\Cat^{+}}(\cF_{\circ}(\bA))\to \bA\]
	of marked categories.  
	
	{Consider the pullback
		\[\xymatrix{
		\bF(\bA) \times_{\bA}\bF(\bA)\ar[r]^-{p_1}\ar[d]_{p_2}&\bF(\bA)\ar[d]^{v_\bA}\\
		\bF(\bA)\ar[r]^-{v_\bA}&\bA	
}\]
We claim that the diagram}
\begin{equation*}
\begin{tikzcd}
\bF(\bA) \times_{\bA}\bF(\bA) \ar[r, shift left, "p_1"]\ar[r, shift right, "p_2"']&\bF(\bA)\ar[r, "v_\bA"]&\bA
\end{tikzcd}
\end{equation*}
is a coequalizer. We have $v_\bA\circ p_1=v_\bA\circ p_2$ by definition. That every morphism $f\colon \bF(\bA)\to \bB$ with $f\circ p_1=f\circ p_2$ factors uniquely through $v_\bA$ follows from the fact that $v_\bA$ is surjective on objects and full.
	
	 	We know that $\bF(\bA)$ is isomorphic to a colimit of a small diagram involving the  
		list of finite categories
	\begin{equation*}
\{ \Free_{\Cat^{+}} (*)\ ,  \Free_{\Cat^{+}}(\bullet \to \bullet),  \Free_{\Cat^{+}}(\bullet\xrightarrow{+} \bullet)\}\ . 
\end{equation*} 
	The fiber product over $\bA$ is not a colimit. But we have a surjection 
	\[v'_\bA=v_{\bF(\bA) \times_{\bA}\bF(\bA)}\colon \bF( \bF(\bA) \times_{\bA}\bF(\bA))\to  \bF(\bA) \times_{\bA}\bF(\bA)\]
	and therefore a  {coequalizer diagram}
\begin{equation*}
\begin{tikzcd}
\bF(\bF(\bA) \times_{\bA}\bF(\bA)) \ar[r, shift left, "p_1\circ v'_\bA"]\ar[r, shift right, "p_2\circ v'_\bA"']&\bF(\bA)\ar[r, "v_\bA"]&\bA.
\end{tikzcd}
\end{equation*}	
	The marked category $ \bF( \bF(\bA) \times_{\bA}\bF(\bA))$ is again a colimit of a diagram involving the generators in the list above. Hence $\bA$ itself is a colimit of a diagram built from this list.	
	
	 	 A similar argument applies in the  case $ \preAdd^{+}$.
		In this case we must replace $\cF_{\circ}$ by $\cF_{\circ}\circ \cF_{\Z}$ and $\Free_{\Cat^{+}}$ by
		$\Lin_{\Z}\circ \Free_{\Cat^{+}}$. The list of generators is 
\begin{equation*}
\mathclap{
\{ \Lin_{\Z}(\Free_{\Cat} (*))\ ,  \Lin_{\Z}(\Free_{\Cat}(\bullet \to \bullet))\ ,  \Lin_{\Z}(\Free_{\Cat}(\bullet\xrightarrow{+} \bullet))\}\ . 
}
\end{equation*}
These categories are again compact since they have finitely many objects and their morphism groups are 	finitely generated.
\end{proof}

\begin{rem}
In order to show that $\Cat$ and $\preAdd$ are locally presentable one argues similarly using the category of directed graphs $\Quivers$ and the adjunctions
\[\Free_{\Cat} \colon \Quivers\leftrightarrows \Cat \colon \cF_{\circ}\ , \quad \Lin_{\Z}\circ \Free_{\Cat} \colon \Quivers\leftrightarrows \preAdd \colon \cF_{\Z}\circ \cF_{\circ}\ .\qedhere\]
\end{rem}

\subsection{(Marked) additive categories as fibrant objects}\label{gioegregreget34t34t34t34t}

In \cref{vgioeoerberebg} we have shown that the simplicial categories $\preAdd$ and $\preAdd^{+}$ are locally presentable and have a  simplicial, cofibrantly generated model category structures. In the present section we introduce Bousfield localizations of these categories whose categories of fibrant objects are exactly the additive categories or marked additive categories.

Let $\bA$ be a pre-additive category.
\begin{ddd}\label{rioehgjoifgregregegergeg}
	We say that $\bA$ is \emph{additive} if $\bA$ has a zero object and the sum, see \cref{vgeroihirovervbervevev}, of any two objects of $\bA$ exists.
\end{ddd}
  
  We let $\Add$ denote the full subcategory of $\preAdd$ of additive categories.
    
  \begin{rem}
 In contrast to being a pre-additive category, being an additive category is a property of a category.  
 In the following we describe the conditions for an additive category just in terms of category language. First of all  we require the existence of a zero object which by definition is an object which is both initial and final.  Furthermore  we require  the  existence of finite products and coproducts, and that the natural transformation
 \[ -\sqcup-  \to  -\times-\]
 of  bifunctors (its definition uses the zero object)  is an isomorphism. This leads naturally to an enrichment over commutative monoids. Finally we require that the morphism monoids are in fact abelian groups. 
 
 A  morphism between additive categories can be characterized as   a functor which preserves products.
 It then preserves sums, zero objects, and the enrichment automatically. Here one can also interchange the roles of sums and products. 
 
  Therefore $\Add$ can be considered as a (non-full) subcategory of $\Cat$.
     \end{rem}

 Let $(\bA,\bA^{+})$ be a marked pre-additive category.
 \begin{ddd}\label{reiuheriververvec}
 $(\bA,\bA^{+})$ is a \emph{marked additive category} if the following conditions are satisfied:
 \begin{enumerate}
 \item The underlying category $\bA$ is additive.
 \item\label{fdblkgjklrgregergergerg} $\bA^{+}$ is closed under  sums.   \qedhere \end{enumerate}
 \end{ddd}
In detail,  Condition \ref{fdblkgjklrgregergergerg}
 means that
 for every two  morphisms $a \colon A\to A^{\prime}$ and $b \colon B\to B^{\prime}$ in $\bA^{+}$ the induced isomorphism $a\oplus b \colon A\oplus B\to A^{\prime}\oplus B^{\prime}$ (for any choice of objects and structure maps representing the sums) also belongs to $\bA^{+}$.
 
In \cref{rgvoihifowefwfwe} below we will discuss a natural example of a marked pre-additive category
in which the Condition \ref{fdblkgjklrgregergergerg} is violated.

 \begin{ex}
 A category $\bC$ with cartesian products can be refined to a symmetric monoidal category with the cartesian symmetric monoidal structure \cite[Sec.~2.4.1]{HA}. In particular we have a functor (uniquely defined up to unique isomorphism)
 \[ -\times- \colon \bC\times \bC \to \bC\ .\]
 This applies to  an additive category $\bA$ where the cartesian product is denoted by $\oplus$.
 We therefore have a sum functor  
 \[ -\oplus - \colon \bA\times \bA\to \bA\ .\]
 Note that $\bA\times \bA$ (the product is taken in $\preAdd$) is naturally an additive category again, and that the sum functor is a morphism of additive categories.

 If $(\bA,\bA^{+})$ is now a marked additive category, then $(\bA,\bA^{+})\times (\bA,\bA^{+})$ (the product is taken in $\preAdd^{+}$)  is marked again, and Condition \ref{reiuheriververvec}.\ref{fdblkgjklrgregergergerg} 
 implies that we also have a functor
 \[ -\oplus - \colon (\bA,\bA^{+})\times (\bA,\bA^{+})\to(\bA,\bA^{+}) \]
 between marked additive categories.
\end{ex}

 We want to reformulate the characterization of (marked) additive categories from \cref{rioehgjoifgregregegergeg} and \cref{reiuheriververvec} as a right-lifting property.
To this end we introduce  the pre-additive categories
$\bS_{\preAdd}$ and $\emptyset_{\preAdd}$ in $\preAdd$ given  as follows:
\begin{enumerate}
	\item The pre-additive category $\bS_{\preAdd}$ has three objects   $1$, $2$, and $S$ 
	and the morphisms are generated by  the morphisms
	 	\[\{ 1\xrightarrow{i_{1}}S ,2\xrightarrow{i_{2}}S, S\xrightarrow{p_{1}}1,S\xrightarrow{p_{2}}2 \}\ .\]
	 subject to the following relations:
	\[ \quad p_{1}\circ i_{1}=\id_{1}\ , \quad p_{2}\circ i_{2}=\id_{2}\ , \quad i_{1}\circ p_{1} +i_{2}\circ p_{2}=\id_{S}\ .\]
	\item  $\emptyset_{\preAdd}$ has one object $0$ and $\Hom_{\emptyset_{\preAdd}}(0,0)=\{\id_0\}$. Note that $\id_0$ is the zero morphism.
 \end{enumerate}

We further define the marked versions
\[ \bS_{\preAdd^{+}}:=\mi(\bS_{\preAdd})\ , \quad \emptyset_{\preAdd^{+}}:=\mi(\emptyset_{\preAdd}) \]
in $\preAdd^{+}$ by marking the identities.

 In the following	 let $\cC$ be a place holder for  $\preAdd$ or $\preAdd^{+} $.

\begin{rem}\label{fiofjoiwefwfwefewfw}   We consider the object $\bS_{\cC}$ of $\cC$.
	Note that the relations $p_{1}\circ i_{2}=0$ and $p_{2}\circ i_{1}=0$ are implied.  The morphisms $p_{1}$, $ p_{2}$   present $S$ as the product of $1$ and $2$, and the morphisms  $i_{1}$ and $i_{2}$ present $S$ as a coproduct of $1$ and $2$. Consequently, $S$ is the sum of the objects $1$ and $2$, see \cref{vgeroihirovervbervevev}.

 	If $\bA$ belongs to $\cC$ and $f \colon \bS_{\cC}\to \bA$ is a morphism, then
   the morphisms $f(p_{1})$, $ f(p_{2})$   present $f(S)$ as the product of $f(1)$ and $f(2)$, and the morphisms $f(i_{1})$, $f(i_{2})$ present $f(S)$ as a coproduct of $f(1)$ and $f(2)$. Hence again, $f(S)$ is the sum of the objects $f(1)$ and $f(2)$.
   
 A functor $\bS_{\cC}\to \bA$ is the same as the choice of two objects $A$, $B$ in $\bA$ together with a representative of the sum $A\oplus B$  and the corresponding structure maps.  
     \end{rem}

\begin{rem}
The object $0$ of $\emptyset_{\cC}$ is a zero object. If $\bA$ belongs to $\cC$ and $f \colon \emptyset_{\cC}\to \bA$ is a morphism, then $f(0)$ is an object satisfying $\id_{f(0)} = 0$. Since $\bA$ is enriched over abelian groups, every object in $\bA$ admits a morphism to $f(0)$ and a morphism from $f(0)$, both of which are necessarily unique. Hence $f(0)$ is a zero object of $\bA$.
In fact, $\emptyset_{\cC}$ is the zero-object classifier in $\cC$.
\end{rem}

Recall the notation introduced in \cref{efuweifo24frergergreg}.
 We let \begin{equation}\label{g4pogkp40gggt}
w \colon \Delta^{0}_{\cC}\sqcup \Delta^{0}_{\cC}\to \bS_{\cC}
\end{equation} be the morphism which  classifies 
the two objects $1$ and $2$. We furthermore let  \begin{equation}\label{g4pogkp40gggt1}v \colon \emptyset\to \emptyset_{\cC}
\end{equation} be the canonical morphism from the initial object of $\cC$.

We now use that $\cC$ is a left-proper (see   \cref{fiowefwefewfewf}), combinatorial simplicial model category (see \cref{vgioeoerberebg}). By \cite[Prop.~A.3.7.3]{htt}, for every set $\cS$ of cofibrations in $\cC$ the left Bousfield localization $L_{\cS}\cC$ (see \cite[Def. 3.3.1]{MR1944041} or \cite[Sec.~A.3.7]{htt} for a definition) exists and is again a  combinatorial simplicial model category. We will consider the set $\cS:=\{v,w\}$ consisting of the cofibrations 
\eqref{g4pogkp40gggt} and \eqref{g4pogkp40gggt1}.

\begin{prop}\label{rigerogergergre}
	The fibrant objects in $L_{\{v,w\}}\cC$ are exactly the (marked) additive categories.
\end{prop}

\begin{proof}  
The fibrant objects in $L_{\{v,w\}}\cC$ are the fibrant objects $\bA$ in $\cC$ which are local for $\{v,w\}$, i.e., for which the maps of simplicial sets  $\Map_{\cC}(v,\bA)$ and $\Map_{\cC}(w,\bA)$ are trivial Kan fibrations, see \cite[Prop.~A.3.7.3(3)]{htt}.

Let $\bA$ be  in $\cC$ and  consider  the lifting problem \begin{equation}\label{ecwcknbweckjwecewcecwec}
\xymatrix{\partial \Delta^{n}\ar[r]\ar[d]&\Map_{\cC}(\bS_{\cC},\bA)\ar[d]^{\Map_{\cC}(w,\bA)}\\\Delta^{n}\ar[r]\ar@{..>}[ur]&\Map_{\cC}( \Delta^{0}_{\cC}\sqcup \Delta^{0}_{\cC},\bA)}\ .
\end{equation}
Since the mapping spaces in $\cC$ are nerves of groupoids they are $2$-coskeletal.
Hence the lifting problem is uniquely solvable for all $n\ge 3$ without any condition on $\bA$.
 It therefore suffices to consider the cases $n=0,1,2$. 
\begin{enumerate}\item[n=0] The outer part of the diagram reflects the choice of two objects in $\bA$, and a lift corresponds to a choice of a sum of these objects together with the corresponding structure maps.
Therefore the lifting problem is solvable if $\bA$ admits sums of pairs of objects.

 \item[n=1]  The outer part of the diagram reflects  the choice of
(marked) isomorphisms $A\to A^{\prime}$ and $B\to B^{\prime}$ in $\bA$
and choices of objects $A\oplus A^{\prime}$ and $B\oplus B^{\prime}$ together with structure maps (inclusions and projections) representing the sums.
The lift then corresponds to the choice of a (marked)  isomorphism $A\oplus A^{\prime}\to B\oplus B^{\prime}$.
In fact such an isomorphism exists (and is actually uniquely determined). In the marked case the fact that the isomorphism is marked is equivalent to the compatibility condition between the sums and the marking required for   a marked additive category.
\item[n=2] The outer part reflects the choice of six objects $A,A',A''$ and $B,B',B''$ together with the choice of objects representing the sums $A\oplus B$, $A'\oplus B'$ and $A''\oplus B''$ together with structure maps and  (marked) isomorphisms   $a \colon A\to A'$, $a' \colon A'\to A''$, $a'' \colon A\to A''$, and $b \colon B\to B'$, $b' \colon B'\to B''$, $b'' \colon B\to B''$  respectively, and compatible (with the structure maps and hence uniquely determined) (marked) isomorphisms    $a\oplus b \colon A\oplus B\to A'\oplus B'$, $a'\oplus b' \colon A'\oplus B'\to A''\oplus B''$, $a''\oplus b'' \colon A\oplus B\to A''\oplus B''$.
Thereby we have the relations $a''=a'\circ a$ and $b''=b'\circ b$.
A lift corresponds to a witness of the fact that  $a''\oplus b''=(a'\oplus b')\circ (a\oplus b)$. 
Hence the lift exists and is unique by the universal properties of the sums.
\end{enumerate}

We have 
\[ \Map_{\cC}(v,\bA) \colon \Map_{\cC}(\emptyset_{\cC},\bA)\to *\ .\]
The domain of this map is the space of zero objects in $\bA$ which is either empty or a contractible Kan complex.  
Consequently, $\Map_{\cC}(v,\bA)$ is a trivial Kan fibration exactly if $\bA$ admits a zero object.  
\end{proof}

\subsection{\texorpdfstring{$\infty$}{infty}-categories of (marked)  pre-additive and additive categories}

In the present paper we use the language of  $\infty$-categories as developed in  \cite{Joyal}, \cite{htt} and \cite{cisin}.
Let 
  $\bC$ be  a   simplicial model category. By  \cite[Thm.~1.3.4.20]{HA}, we have an equivalence of $\infty$-categories  \begin{equation}\label{g4r5g4f34rf3f4f}
\cNerve(\bC^{cf}) \simeq \bC^{c}[W^{-1}]\ ,
\end{equation}
   where  $\cNerve(\bC^{cf})$ is the  coherent nerve of the simplicial category of cofibrant-fibrant objects in $\cC$, and  $\bC^{c}[W^{-1}]$ is the $\infty$-category obtained from  (the nerve of) $\bC^{c}$ by inverting the weak equivalences of the model category structure,  where $\bC^{c}$ denotes the  ordinary category of cofibrant objects of $\bC$.  
 If $\cC$ is in addition combinatorial, then
 $\bC^{c}[W^{-1}]$
 is a presentable $\infty$-category  \cite[Prop.~1.3.4.22]{HA}.  
 
 For the following we assume that $\bC$ is a combinatoral  simplicial model category.
If $L_{\cS} \bC$ is the Bousfield localization of the model category structure on $\bC$ at a set  $\cS$ of morphisms  in $\bC^{cf}$,  and $\cNerve(\bC^{cf})\to L_{\cS}\cNerve(\bC^{cf})$ is the localization    at the same set of morphisms   in the sense of \cite[Def.~5.2.7.2]{htt}, then   using  \cite[Rem.~1.3.4.27]{HA}
we get an equivalence of $\infty$-categories
\[ L_{\cS}\cNerve( \bC^{cf})\simeq  \cNerve((L_{\cS} \bC)^{cf})\ .\]

 We let $W_{\preAdd^{(+)}}$ denote the weak  equivalences in $\preAdd^{(+)}$. Note that in $\preAdd^{(+)}$ all objects are cofibrant and fibrant.
\begin{ddd}
We define the \emph{$\infty$-category of (marked) pre-additve} categories by
\[\preAdd^{(+)}_{\infty}:= \preAdd^{(+)}[W_{\preAdd^{(+)}}^{-1}]\ .\qedhere\]
\end{ddd}
 By a specialization of  \eqref{g4r5g4f34rf3f4f}  we have an equivalence of $\infty$-categories
\begin{equation}\label{ivfou89f43fvfeferferferf}
\cNerve( \preAdd^{(+)}) \simeq \preAdd_{\infty}^{(+)} \ .
\end{equation}
 A weak equivalence between fibrant objects in a Bousfield localization is a weak equivalence in the original model category. Consequently,  a morphism between (marked) additive categories is
a weak equivalence in $L_{\{v,w\}}\cC$ if and only if it is a weak equivalence   in (marked) pre-additive categories.  
 
We let $W_{\Add^{(+)}} $ denote the weak  equivalences in the Bousfield localization $L_{\{v,w\}}\preAdd^{(+)}$.
 \begin{ddd}
 We define the \emph{$\infty$-category of  (marked) additive categories} by
\[\Add^{(+)}_{\infty}:= \preAdd^{(+)}[W_{\Add^{(+)}}^{-1}]\ .\qedhere\]
 \end{ddd}
 By  specialization of  \eqref{g4r5g4f34rf3f4f},   we then have an equivalence of $\infty$-categories  \begin{equation}\label{ivfou89f43fvfeferferferf1}
 \cNerve(\Add^{(+)})\simeq \Add^{(+)}_{\infty}\ .
\end{equation}

 \begin{rem}The equivalences  \eqref{ivfou89f43fvfeferferferf} and  \eqref{ivfou89f43fvfeferferferf1}       can be shown directly using 
 \cite[Prop.~1.3.4.7]{HA}. Indeed, the categories $\preAdd^{(+)}$ and $\Add^{(+)}$  are enriched in groupoids and therefore fibrant simplicial categories. The interval object of $\bA$ is given by $\bA^{\Delta^{1}}$.
 In the case of (marked) additive categories we must observe that $\bA^{\Delta^{1}}$ is again (marked) additive.
\end{rem}

\begin{kor} \label{fiowefwefwfwf} \begin{enumerate} \item
The $\infty$-categories $\preAdd^{(+)}_{\infty}$ and 
	$ \Add^{(+)}_{\infty}$  are presentable. 
	\item We have an adjunction
	\begin{equation}\label{vevnkvnrekovefveveervv}
 L_{\oplus} \colon \preAdd^{(+)}_{\infty}\leftrightarrows  \Add^{(+)}_{\infty} \colon \cF_{\oplus}\ ,
\end{equation} 
	where $\cF_{\oplus}$ is the inclusion of a full subcategory. \end{enumerate}
\end{kor}
The functor $L_{\oplus}$ is the additive completion functor.

  In the following  $\cC$ is a placeholder for $\Cat^{(+)}$, $\Add^{(+)}$ or $\preAdd^{(+)}$.

The category $\cC$ can be considered as a category enriched in  groupoids and   therefore as a  strict $(2,1)$-category which will be denoted by  $\cC_{(2,1)}$. A   strict  $(2,1)$-category gives rise to an $\infty$-category as follows.  We first apply the usual nerve functor to the morphism categories of $\cC_{(2,1)}$ and obtain a category  enriched in Kan complexes. Then we   apply the coherent nerve functor and get a quasi-category which we will denote by $\Nerve_{2}(\cC_{(2,1)})$. The obvious  functor $\Nerve(\cC_{(1,1)})\to \Nerve_{2}(\cC_{(2,1)})$ (where $\cC_{(1,1)}$ denotes the underlying ordinary category of $\cC$) sends equivalences to equivalences and therefore descends to a functor
 \begin{equation}\label{rgieogrgregegegerg}
 \cC_{\infty}\to \Nerve_{2}(\cC_{(2,1)})\ .
\end{equation}
 \begin{prop}
 	\label{prop:2nerv}
 The functor \eqref{rgieogrgregegegerg} is an equivalence.   \end{prop}
 \begin{proof}
 	Note that $N_2(\cC_{(2,1)})$ and $N^{coh}(\cC)$ are isomorphic by the definition of the simplicial enrichtment of $\cC$.
 
 We consider the following commuting diagram of quasi-categories
 \[ \xymatrix{&\Nerve(\cC_{(1,1)})\ar[ld]_{\ell_{\cC}}\ar[dr] \ar[d]^{!}&&\\\cC_{\infty}\ar[r]_-{\simeq}^-{!!}&\Nerve^{coh}(\cC)\ar[r]_-{\cong}&\Nerve_{2}(\cC_{(2,1)})}\ .\]

 The left triangle commutes since the morphism marked by $!$ is an explicit model of the localization morphism,
  where we use
   \eqref{ivfou89f43fvfeferferferf}  (or  \eqref{ivfou89f43fvfeferferferf1}, depending on the case) for the equivalence marked by $!!$.   
 The lower composition is then an explicit model of  \eqref{rgieogrgregegegerg}.
 \end{proof}

\section{Applications}\label{erigjieorgergree34t34t34t}

\subsection{Localization preserves products} \label{efweoifoewfewfewf3r323r2r}
 
 We show that the localizations
 \[ \ell_{\cC^{(+)}} \colon \cC^{(+)}\to \cC^{(+)}_{\infty} \]
 for $\cC$ in $\{\Cat\ ,\preAdd\ ,\Add\}$
 preserve products. 
 
 Let $I$ be a set. 
 Then we consider  the functor
 \[ \ell_{I,\cC}\colon \cC^{I}\to   \cC_{\infty}^{I} \]
 defined by post-composition with $\ell_{\cC}$.
 For every category $\bC$ with products we  have a functor $\prod_{I} \colon \bC^{I}\to \bC$. We apply this to $\bC=\cC$ and $\bC=\cC_{\infty}$.
 
 \begin{prop}\label{wefiojewwefewf43t546466}
 	We have an equivalence of functors
 	\[ \ell_{\cC}\circ \prod_{I} \xrightarrow{\simeq}
 	\prod_{I}\ell_{I,\cC} \colon \cC^{I}\to \cC_{\infty}\ . \]
 \end{prop}
 \begin{proof}  
 	We start with the case $\cC=\preAdd^{(+)}$ or $\cC=\Cat^{(+)}$. We use that $\cC$ has a combinatorial model category structure in which all objects are cofibrant and fibrant. It is a general fact, that in this case the localization $\ell \colon \cC\to\cC_{\infty}$ preserves products. Here is the (probably much too complicated) argument. 
 	We can 
 	consider the injective model category structure on the diagram category $ \cC^{I}$. Since $I$ is discrete 
 	one easily observes that    all objects in this diagram category are  fibrant  again. 
 	So we can take the identity as a fibrant replacement functor for $\cC^{I}$. This gives  the equivalence
 	\[ \ell_{\cC} \circ \prod_{I} \xrightarrow{\simeq} \prod_{I}\ell_{I,\cC}\ ,\]
 	(e.g. by specializing  \cite[Prop.~13.5]{bunke}).
 	
 	In order to deduce the assertion for additive categories we consider the inclusion functor
 	$\cF_{\oplus,1} \colon \Add^{(+)}\to \preAdd^{(+)}$. This functor preserves weak equivalences and therefore descends essentially uniquely to the functor  $\cF_{\oplus}$ in \eqref{vevnkvnrekovefveveervv} such that
 	\[ \cF_{\oplus}\circ \ell_{\Add^{(+)}}\simeq \ell_{\preAdd^{(+)}}\circ \cF_{\oplus,1}\ .\]
 	The functor $\cF_{\oplus}$ is a right-adjoint which preserves and detects limits. 
 	We do not claim that $\cF_{\oplus,1}$ is a right-adjoint, but it clearly preserves products by inspection.
 	We let $\cF_{I,\oplus,1}$ and $\cF_{I,\oplus}$ be the factorwise application of $\cF_{\oplus,1}$ and $\cF_{\oplus}$.
 	With this notation we have an equivalence
 	\[ \cF_{\oplus,1}\circ \prod_{I}\cong  \prod_{I} \circ  \cF_{I,\oplus,1}\ . \]
 	The assertion in the case $\cC=\Add^{(+)}$ now follows from the chain of equivalences \begin{eqnarray*}
 		\cF_{\oplus}\circ \ell_{\Add^{(+)}}\circ \prod_{I}&\simeq&
 		\ell_{\preAdd^{(+)}}\circ \cF_{\oplus,1}\circ \prod_{I}\\&\simeq&
 		\ell_{\preAdd^{(+)}}\circ \prod_{I}  \circ \cF_{I,\oplus,1}\\&\simeq&
 		\prod_{I}     \circ \ell_{I,\preAdd^{(+)}}   \circ \cF_{I,\oplus,1}\\&\simeq&
 		\prod_{I}     \circ \cF_{I,\oplus}\circ \ell_{I,\Add^{(+)}}    \\&\simeq&
 		\cF_{\oplus}     \circ  \prod_{I}  \circ \ell_{I,\Add^{(+)}}   
 	\end{eqnarray*}
 	by removing $\cF_{\oplus}$.
 \end{proof}

\subsection{Rings and Modules}\label{rgiuerhgweergergergeg}

A unital ring $R$ can be  considered as a pre-additive category $\bR$ with one object $*$ and ring of endomorphisms $\Hom_{\bR}(*,*):=R$.
The category of finitely generated free $R$-modules $\Mod^{\fg ,\free}(R)$ is an additive category.
We have a canonical functor
$\bR\to \Mod^{\fg ,\free}(R)$ sending $*$ to $R$
which presents  $\Mod^{\fg ,\free}(R)$ as the additive completion of $\bR$.
This fact is well-known, see e.g. \cite[Sec.~2]{davis_lueck}. 
In the following we provide a precise formulation  using the language of $\infty$-categories.

Recall the sum-completion  functor $L_{\oplus}$ from \cref{fiowefwefwfwf}.

\begin{prop}\label{gueiurgrgerger}
	The morphism of pre-additive categories $\bR\to \Mod^{\fg ,\free}(R) $ induces an equivalence
	\[ L_{\oplus} (\ell_{\preAdd}(\bR))\simeq \ell_{\Add}(\Mod^{\fg ,\free}(R))\ .\]
\end{prop}
\begin{proof}
	We must show that
	\[ \Map_{\preAdd_{\infty}}(\ell_\preAdd(\Mod^{\fg ,\free}(R)),\ell_\preAdd(\bB))\to \Map_{\preAdd_{\infty}}(\ell_\preAdd(\bR),\ell_\preAdd(\bB)) \]
	is an equivalence for every additive category $\bB$. In view of \eqref{ivfou89f43fvfeferferferf}, this is equivalent to the fact that
	\[ \Map_{\preAdd }( \Mod^{\fg ,\free}(R) , \bB) \to \Map_{\preAdd }( \bR , \bB ) \]
	is a trivial Kan fibration. Here we use that by \eqref{ivfou89f43fvfeferferferf} the mapping spaces in $\preAdd_{\infty}$ are represented by the simplicial mapping spaces in $\preAdd$, see 
	 \cite[Sec.~2.2.2]{htt}. The proof is very similar to the proof of \cref{rigerogergergre}.
	We must check the lifting property against the inclusions $\partial \Delta^{n}\to \Delta^{n}$. Again we must only consider the case $n\le 2$.  
	\begin{enumerate}
		\item[n=0] A functor $\bR\to \bB$ (sending $*$ to  an $R$-module $B$) determines a functor
		\[\Mod^{\fg ,\free}(R)\to \bB\]
		which sends $R^{k}$ to $B^{\oplus k}$.
		\item[n=1] An isomorphism of functors $\bR\to \bB$ given by an isomorphism of objects $f \colon B\to B^{\prime}$ which is compatible with the $R$-module structures induces an isomorphism of induced functors
		$ \Mod^{\fg ,\free}(R)\to \bB$ which on $R^{k}$ is given by $\oplus_{k} f \colon B^{\oplus k}\to B^{\prime,\oplus k}$.
		\item[n=2] The existence of the lift expresses the naturality of the isomorphisms obtained in the case $n=1$.\qedhere
	\end{enumerate}
\end{proof}

In order to understand the category of finitely generated projective modules $\Mod^{\fg ,\proj}(R)$ and the morphism $\bR\to \Mod^{\fg ,\proj}(R)$ in a similar manner we must consider idempotent-complete additive categories.
Let $\bA$ be an   additive category.
\begin{ddd}\label{vijoefvfbevev}
	$\bA$ is \emph{idempotent complete} if for every object $A$ in $\bA$ and projection $e$ in $\End_{\bA}(A)$ there exists an   isomorphism $A\cong e(A)\oplus  e(A)^{\perp}$ such that $e(A)$ and $e(A)^{\perp}$ are  images of $e$ and $\id_{A}-e$.
	\end{ddd}
The last part of this definition more precisely means that  there exist  morphisms $e(A)\to A$ and $e(A)^{\perp}\to A$ such  that the diagrams
\[\xymatrix{A\ar[d]^{e}&\ar[l]_(0.7){\cong}e(A)\oplus e(A)^{\perp}\ar[d]^{\pr_{e(A)}}&e(A)\oplus e(A)^{\perp}\ar[d]^{\pr_{e(A)^{\perp}} }\ar[r]^(0.7){\cong}&A\ar[d]^{\id_{A}- e} \\ A&\ar[l]_(0.7){\cong} e(A)\oplus e(A)^{\perp}&e(A)\oplus e(A)^{\perp}\ar[r]^(0.7){\cong}&A}\]
commute.

Let now $\bA$ be a  marked additive category.
\begin{ddd}
$\bA$ is \emph{idempotent complete} if the underlying additive category $\cF_{+}(\bA)$ is idempotent complete (\cref{vijoefvfbevev}), and if in addition for every $A$ in $\bA$,  every projection $e$ on $A$, and every
marked isomorphism  $f \colon A\to A^{\prime}$ the induced isomorphism $e(A)\to e^{\prime}(A')$ is marked, where   $e^{\prime}:=f\circ e\circ f^{-1}$. 
\end{ddd}

We let $\Add^{(+),\idem }$ be the full subcategory of $\Add^{(+)}$ of idempotent complete small (marked) additive categories.

We can characterize idempotent completeness of a marked additive category as a lifting property. To this end we consider the following pre-additive category $\bE_{\preAdd}$:
\begin{enumerate}
	\item $\bE_{\preAdd}$ has the object $*$.
	\item The morphisms of $\bE_{\preAdd}$ are generated by $\id_{*}$ and $e$ subject to the relation $e^{2}=e$.\end{enumerate}
We then consider the functor \begin{equation}\label{rev3rg3rgvervger43}
u \colon \bE_{\preAdd}\to \bS_{\preAdd}
\end{equation}
(see \cref{gioegregreget34t34t34t34t} for $\bS_{\preAdd}$)
which sends $*$ to $S$ and $e$ to $ i_{1}\circ p_{1}$. In the marked case we consider
\[ u \colon \bE_{\preAdd^{+}}\to \bS_{\preAdd^{+}} \]
obtained from \eqref{rev3rg3rgvervger43} by applying the functor $\mi$ marking the identities.
Then one checks:
\begin{lem}
	A (marked) additive category $\bA$ is idempotent complete if and only if it is local with respect to the map $u$. 
\end{lem}
\begin{proof}
	The proof is similar to the proof of \cref{rigerogergergre}.\end{proof}

\begin{kor}
	The fibrant objects in  the Bousfield localization $L_{\{u,v,w\}}\preAdd^{(+)}$ are exactly the idempotent-complete small (marked) additive categories.
\end{kor}

 {We consider the equivalences
$W_{\Add^{(+),\idem }}$ in the Bousfield localization $L_{\{u,v,w\}}\preAdd^{(+)}$ and the $\infty$-category
\[\Add^{(+),\idem }_{\infty}:= \preAdd^{(+)}[W_{\Add^{(+),\idem }}^{-1}]\ .\]
Using \eqref{g4r5g4f34rf3f4f},} we have an equivalence \begin{equation}\label{ivfou89f43fvfeferferferf11}
\cNerve({\Add^{(+),\idem}})\simeq\Add^{(+),\idem }_{\infty} \ .
\end{equation}
We obtain the analog of \cref{fiowefwefwfwf}.
\begin{kor}
	\begin{enumerate}
		\item The $\infty$-category $\Add^{(+),\idem }_{\infty}$ is presentable.
		\item We have an adjunction
		\[ L_{\idem} \colon \Add^{(+)}_{\infty}\leftrightarrows \Add^{(+),\idem }_{\infty} \colon\cF_{\idem} \]
		where $ \cF_{\idem}$ is the inclusion and $L_{\idem}$ is the idempotent completion functor.
		\item We have an adjunction
		\[ L_{\oplus,\idem } \colon \preAdd_{\infty}^{(+)}\leftrightarrows  \Add^{(+),\idem }_{\infty} \colon \cF_{\oplus,\idem } \]
		where $\cF_{\oplus,\idem }\simeq \cF_{\oplus}\circ  \cF_{\idem}$ and $L_{\oplus,\idem }\simeq L_{\idem}\circ L_{\oplus}$.
	\end{enumerate}
\end{kor}

\begin{prop}\label{vgirejgoiergergergergregergergerg}The morphism of pre-additive  categories $\bR\to \Mod^{\fg ,\proj}(R)$ induces an
	equivalence 
	$L_{\oplus,\idem } (\ell_{\preAdd}(\bR))\simeq \ell_{\Add^{\idem}}(\Mod^{\fg ,\proj}(R))$.
\end{prop}
\begin{proof}
	The proof is similar to \cref{gueiurgrgerger}.
\end{proof}The following is a precise version of the assertion that $\Mod^{\fg ,\proj}(R)$ is the idempotent completion of $\Mod^{\fg ,\free}(R)$.
\begin{kor}The morphism of additive categories $\Mod^{\fg ,\free}(R)\to \Mod^{\fg ,\proj}(R)$ induces
	an equivalence
	\[ \ell_{\Add^{\idem}}(\Mod^{\fg ,\proj}(R))\simeq L_{\idem}( \ell_{\Add}(\Mod^{\fg ,\free}(R)))\ .\]
\end{kor}

\subsection{\texorpdfstring{$G$}{G}-coinvariants}\label{erbgkioergergergegreg}

Let $G$ be a group. In this subsection we want to calculate explicitly the homotopy $G$-orbits of pre-additive categories with trivial $G$-action. The precise formulation of the result is \cref{weoijoijvu9bewewfewfwef}. We then discuss applications to group rings.
 
By $BG$ we denote the groupoid with one object $*$ and group of automorphisms $G$. The functor category $\Fun(BG,\bC)$ is the category of objects in $\bC$ with $G$-action and equivariant morphisms. 
The underlying object or morphism of an object or morphism in $\Fun(BG,\bC)$ is the evaluation of the functor or morphism at $*$.
 
If $\bI$ is a category and $F \colon \bC\to \bD$ is a functor, then we will use the notation
 \begin{equation}\label{f34iuh4fiuhrif894r43r34r3434r34r34r31}
 F_{\bI} \colon \Fun(\bI,\bC)\to \Fun(\bI,\bD)\end{equation} for the functor defined by post-composition with $F$.

 We consider a (marked) preadditive category $\bA$. It gives rise to a constant functor
 $\underline{\bA}$ in $\Fun(BG,\preAdd^{(+)})$ and hence to an object
 $\ell_{\preAdd^{(+)},BG}(\underline{\bA})$ in $\Fun(BG,\preAdd^{(+)}_{\infty})$.
  
 Since the $\infty$-category $\preAdd^{(+)}_{\infty}$ is presentable, it is cocomplete and the colimit in the following theorem exists. Recall the functor $- \sharp -$ from \cref{defn_functor_sharp}.
 
\begin{theorem}\label{weoijoijvu9bewewfewfwef}
 We have a natural equivalence
\[ \colim_{ BG}\ell_{\preAdd^{(+)},BG}(\underline{\bA})\simeq \ell_{\preAdd^{(+)}} (\bA\sharp BG)\ .\]
\end{theorem}

\begin{rem}
Note that the order of taking the colimit and the localization is relevant.
Indeed,  we have  $\colim_{BG} \underline{\bA}\cong \bA$ and therefore
$\ell_{\preAdd^{(+)}}(\colim_{BG} \underline{\bA})\simeq \ell_{\preAdd^{(+)}}(\bA)$.
\end{rem}

\begin{rem}\label{ergerg34t3t3434t3}
Note that the unmarked version of \cref{weoijoijvu9bewewfewfwef}
can be deduced from the marked version using the functor $\ma$ introduced in \eqref{f3rfkj34nfkjf3f3f3f43f}. 
\end{rem}

In order to avoid case distinctions, we will formulate the details of the proof in the marked case. The unmarked case can be shown similarly, or alternatively deduced formally from the marked case as noted in \cref{ergerg34t3t3434t3}.

Since $\preAdd^{+}$ has a  cofibrantly generated model category structure, the projective model category structure on $\Fun(BG,\preAdd^{+})$ exists \cite[Thm.~11.6.1]{MR1944041}.
For every cofibrant replacement functor $l \colon L \to \id_{\Fun(BG,\preAdd^{+})}$ for this projective model category structure we have an equivalence  \begin{equation}\label{fwreoiuh24iufhi3rfwrfwefwef}
\ell_{\preAdd^{+}}\circ \colim_{BG}\circ L\simeq \colim_{BG}\circ \ell_{\preAdd^{+},BG}
\end{equation}
of functors from $\Fun(BG,\preAdd^{+})$ to $\preAdd^{+}_{\infty}$, 
see e.g. \cite[Prop.~15.3]{bunke} for an argument.

We derive the formula asserted in \cref{weoijoijvu9bewewfewfwef} by considering a particular choice of a 
  cofibrant replacement functor. 
\begin{ddd}\label{rwgioorgfgwergwf}Let $\tilde G$  in $\Fun(BG,\Groupoids)$ be the groupoid with $G$-action   given as follows:
\begin{enumerate}
\item The objects of $\tilde G$ are the elements of $G$.
\item For every pair of of objects $g,g^{\prime}$ there is a unique morphism $g\to g^{\prime}$.
\item The group $G$ acts on $\tilde G$ by left-multiplication.
\end{enumerate}
The $G$-groupoid $\tilde G$ is often called the transport groupoid of $G$.\end{ddd}

We now define the functor
 \[ L:=-\sharp \tilde G \colon \Fun(BG,\preAdd^{+}) \to \Fun(BG,\preAdd^{+}) \]
 (more precisely  $L(\bD)$ is the $G$-object  obtained from the $G\times G$-object $\bD\sharp \tilde G$ in $\preAdd^{+}$ by restriction of the action along the diagonal $G\to G\times G$).
 We have a natural transformation 
 $L\to \id$ induced by  the morphism of $G$-groupoids $\tilde G\to \underline{\Delta^{0}_{\Cat}}$, where we use the canonical isomorphism  $\bD\sharp \underline{\Delta^{0}_{\Cat}}\cong \bD$. 

\begin{lem}\label{lem:cofibrant-replacement}
The functor $L$ together with the transformation $L\to \id$ 
is a cofibrant replacement functor for the projective model category structure on $\Fun(BG,\preAdd^{+})$.  \end{lem} 
\begin{proof} Since $\Res^{G}_{\{1\}}(\tilde G)\to \Delta^{0}_{\Cat}$ is an (non-equivariant) equivalence of groupoids and  for every  object $\bA$ in  $  \preAdd^{+} $  the functor $\bA\sharp - \colon \Groupoids\to \preAdd^{+}$ preserves  equivalences (see the proof of \cref{reiofweiofweewf}),  the morphism $\bD\sharp \tilde G\to \bD$  is a weak equivalence in the projective model category structure on $\Fun(BG,\preAdd^{+})$ for every $\bD$ in $\Fun(BG,\preAdd^{+})$.

We must show that
$L(\bD)$ is cofibrant. To this end we consider the lifting problem 
\[\xymatrix{\emptyset\ar[r]\ar[d]&\bA\ar[d]^{f} \\ \bD\sharp \tilde  G\ar[r]^{u}\ar@{-->}[ur]^{c}&\bB} \]
where $f$ is a trivial fibration in $\preAdd^{+}$. Since $f$ is surjective on objects
we can find an inverse marked equivalence (possibly non-equivariant) $g \colon \bB\to \bA$ for $f$ such that $f\circ g=\id_{\bB}$. The map
$\bD\sharp \{1\}\xrightarrow{u_{|\bD\sharp \{1\}}} \bB\xrightarrow{g}\bA$ can be uniquely extended to an equivariant morphism $c$ which is the desired lift. \end{proof}

\begin{proof}[Proof of \cref{weoijoijvu9bewewfewfwef}]
 According to \eqref{fwreoiuh24iufhi3rfwrfwefwef} and \cref{lem:cofibrant-replacement}, we must calculate the object
 \[ \colim_{BG} L(\underline{\bA})\cong \colim_{BG} (\underline{\bA}\sharp \tilde G) \]
 for an object $\bA$ of $\preAdd^{+}$.
 To this end, we note that for a fixed marked pre-additive category $\bD$, we have by \eqref{egiherigergregergr1} an adjunction
 \[ \bD\sharp- \colon \Groupoids\leftrightarrows \preAdd^{+} \colon \Fun^{+}_{\preAdd^{+}}(\bD,-)\ .\]
 Since $\bD\sharp-$ is a left-adjoint, it commutes with colimits.
 Consequently, we get \begin{equation}\label{rverv43rgfrrerg}
 \colim_{BG} (\underline{\bA}\sharp \tilde G)\simeq \bA\sharp  \colim_{BG}  \tilde G\ .
\end{equation}
The assertion of \cref{weoijoijvu9bewewfewfwef} now follows from a combination of the relations \eqref{fwrefwfkj2nirkfrwefwfw}, \eqref{rverv43rgfrrerg} and    \eqref{fwreoiuh24iufhi3rfwrfwefwef}.
\end{proof}

Let $R$ be a unital ring.
By $R[G]$ we denote the group ring of $G$ with coefficients in $R$. 
Recall from \cref{rgiuerhgweergergergeg} that we can consider unital rings as pre-additive categories which will be denoted by the corresponding bold-face letters.
\begin{lem}\label{ergioergreger34t3t}
We have an equivalence
\[ \colim_{BG} \ell_{\preAdd^{},BG}(\underline{\bR})\simeq \ell_{\preAdd^{}}(\mathbf{R[G]})\ .\]
\end{lem}
\begin{proof}
	By \cref{weoijoijvu9bewewfewfwef}, we have an equivalence
	\[ \colim_{ BG} \ell_{\preAdd,BG}(\underline{\bR})\simeq \ell_{\preAdd}(\bR\sharp BG)\ .\]
Unfolding  the definitions (see e.g.~\cref{defn_functor_sharp}) we observe that
$\bR\sharp BG$ has one object, and its ring of endomorphisms  is given by
$R\otimes_{\Z }\Z[G]\cong R[G]$.
\end{proof}

\begin{lem}\label{rierhigregegerg43t34t34t}
We have equivalences
\[ \colim_{BG}\ell_{\preAdd,BG}( \underline{  \Mod^{\fg ,\free}(R)})\simeq \ell_{\preAdd}(\Mod^{\fg ,\free}(R[G])) \]
and
\[ \colim_{BG}\ell_{\preAdd,BG} (\underline{ \Mod^{\fg ,\proj}(R)})\simeq \ell_{\preAdd}(\Mod^{\fg ,\proj}(R[G])) \]
\end{lem}

\begin{proof}
By \cref{gueiurgrgerger},
we have an equivalence
\[ \colim_{BG} \ell_{\Add,BG}(\underline{\Mod^{\fg ,\free}(R)})\simeq\colim_{BG}  L_{ \oplus,BG}( \ell_{\preAdd,BG}(\underline{\bR }))\ .\]
Since $L_{\oplus}$ is a left-adjoint, it commutes with colimits.
Therefore,
\[ \colim_{BG}  L_{ \oplus,BG}( \ell_{\preAdd,BG}(\underline{\bR }))\simeq L_{\oplus}(\colim_{ BG}\ell_{\preAdd,BG}(\underline{\bR})) \ .\]
By \cref{ergioergreger34t3t}, we have the equivalence
\[ L_{\oplus}(\colim_{BG}\ell_{\preAdd,BG}(\underline{\bR})) \simeq L_{\oplus} (\ell_{\preAdd}(\mathbf{R[G])}) \ .\]
Finally, by \cref{gueiurgrgerger} again  
\[ L_{\oplus} (\ell_{\preAdd}(\mathbf{R[G])})\simeq \ell_{\preAdd}(\Mod^{\fg ,\free}(R[G]))\ .\]
The second equivalence is shown similarly, using \cref{vgirejgoiergergergergregergergerg} and $L_{\oplus,\idem }$ instead of \cref{gueiurgrgerger}  and $L_{\oplus}$.
\end{proof}

\begin{ex}
A unital ring $R$ gives rise to two canonical marked preadditive categories $\mi(\bR)$ (only the identity is marked) and $\ma(\bR)$ (all units are marked). 
Then
\[ \colim_{BG}\ell_{\preAdd^{+},BG}(\mi(\bR))\simeq \ell_{\preAdd^{+}}(\mathbf{R[G]}^{\can_{G}})\ ,\]
where the marked isomorphisms in $\mathbf{R[G]}^{\can_{G}}$ are the elements of $G$ (canonically considered as elements in $R[G]$).
In contrast,
\[ \colim_{BG}\ell_{\preAdd^{+},BG}(\ma(\bR))=\ell_{\preAdd^{+}}(\mathbf{R[G]}^{\can})\ ,\]
where  the marked isomorphisms in   $\mathbf{R[G]^{\can}}$  are the canonical units in $R[G]$, i.e., the elements of the form $ug$ for a unit $u$ of $R$ and an element $g$ of $G$. 
 \end{ex}

Let us now use the general machine in order to construct interesting functors on the   orbit category 
 $G\Orb$   of $G$. The group $G$ with the left action is an object of $G\Orb$.
Since the right action of $G$ on itself implements an isomorphism $\End_{G\Orb}(G)\cong G$, we get a fully faithful functor
\begin{equation}
\label{eq_BG_GOrb}
i \colon BG\to G\Orb\ .
\end{equation} 
If $\bC$ is a presentable $\infty$-category, then we have an adjunction
\begin{equation}
\label{eq_BG_GOrb_Kan}
i_{!} \colon \Fun(BG,\bC)\leftrightarrows \Fun(G\Orb,\bC) \colon i^{*}\ .
\end{equation}
The functor $i_{!}$ is the left Kan extension functor along $i$.
We now consider the composition
\begin{align}
\preAdd & \quad \xrightarrow{\quad\mathclap{\underline{(-)}}\quad} \quad \Fun(BG,\preAdd)\notag\\
& \quad \xrightarrow{\quad\mathclap{\ell_{\preAdd,BG}}\quad} \quad \Fun(BG,\preAdd_{\infty})\notag\\
& \quad \xrightarrow{\quad\mathclap{i_{!}}\quad} \quad \Fun(G\Orb,\preAdd_{\infty})\label{revelrnjjkrnfkjervervverveverv}
\end{align}
which we denote by $J^G$.
 
We are interested in the calculation of the value  $J^{G}(\bA)(G/H)$ for
a subgroup $H$.

Let $\bA$ be a pre-additive category.
\begin{lem}\label{vfuiheiwufewfewfefewfewf}
We have an equivalence
\[ J^{G}(\bA)(G/H)\simeq \ell_{\preAdd}( \bA\sharp BH)\ .\]
\end{lem}
\begin{proof}
The functor $S\mapsto (G\times_{H}S\to G/H)$ induces an equivalence of categories
$H\Orb\xrightarrow{\simeq} G\Orb/(G/H)$
which restricts to an equivalence
\begin{equation}\label{welrkerioerrfwfwefewf}
BH\simeq i/(G/H)\ ,
\end{equation}
where $i/(G/H)$ denotes the slice of $i \colon BG\to G\Orb$ over $G/H$.
Using the pointwise formula for the left Kan extension functor $i_{!}$ at the equivalence marked by $!$ we get
\begin{eqnarray*}
J^{G}(\bA)(G/H)&\simeq&i_{!}(\ell_{\preAdd,BG}(\underline{\bA}))(G/H)\\&\stackrel{!}{\simeq}&
\colim_{(i(*)\to G/H)\in i/(G/H)}  \ell_{\preAdd,BG}(\underline{\bA})(*)\\
&\stackrel{!!}{\simeq}&\colim_{BH} \ell_{\preAdd,BH}(\underline{\bA}) \\
&\stackrel{\cref{weoijoijvu9bewewfewfwef}}{\simeq} &\ell_{\preAdd}(\bA\sharp BH)\ ,
\end{eqnarray*}
where at $!!$ we use \eqref{welrkerioerrfwfwefewf} and
that the argument of the colimit is a constant functor.
\end{proof}

The case $\bA:=\bR$ for a ring $R$ leads to a functor
\[ J^{G}(\bR) \colon G\Orb \to \preAdd_{\infty} \] whose value at
$G/H$ is given by $J^{G}(\bR)(G/H)\simeq \ell_{\preAdd}(\mathbf{R[H]})$.
If we postcompose by $L_{\oplus}$ and use \cref{gueiurgrgerger}, then we get a functor
\[ L_{\oplus,G\Orb}\circ 
J^{G}(\bR) \colon G\Orb \to \Add_{\infty} \]
with values $L_{\oplus,G\Orb}\circ 
J^{G}(\bR)(G/H)\simeq \ell_{\Add}(\Mod^{\fg ,\free}R[H])$.
The composition
\[ K_{G\Orb}\circ L_{\oplus,G\Orb}\circ 
J^{G}(\bR) \colon G\Orb\to \Sp \]
therefore has the same values as the functor
representing the equivariant $K$-homology with $R$-coefficients constructed by \cite{davis_lueck}.

\subsection{\texorpdfstring{$G$}{G}-invariants}\label{gijeriogjeroigergregeg}

 Let $G$ be a group.
In this section we calculate the homotopy $G$-invariants of marked pre-additive categories with $G$-action.
The precise formulation is \cref{rgier9oger}.  

  Let 
  $\bA$ be an object of $\Fun(BG,\preAdd^{(+)})$, i.e., a (marked) pre-additive category with $G$-action.
  \begin{ddd}\label{ddd_hat_A_G}
  We define a (marked) pre-additive category $\hat \bA^{G}$ as follows:
  \begin{enumerate}
		\item The objects of $\hat \bA^{G}$ are pairs $(A,\rho)$ of  an  object  $A$ of $\bA$ 
		and a  collection    
		$\rho:=(\rho(g))_{g\in G}$, where $\rho(g) \colon A\to g(A)$ is a (marked) isomorphism in $\bA$ and the equality
		\[g(\rho(h)) \circ \rho(g)=\rho(hg)\]
		holds true
		for all pairs $g,h$ in $G$.
		\item\label{guerioggergerg} The morphisms $(A,\rho)\to (A^{\prime},\rho^{\prime})$ in $\hat \bA^{G}$ are morphisms $a \colon A\to A^{\prime}$ in $\bA$ such that the equality $g(a)\circ \rho(g)=\rho^{\prime}(g)\circ a$ holds true for all $g$ in $G$.
		\item The enrichment of $\hat \bA^{G}$ over abelian groups  is inherited from the enrichment of $\bA$.
		\item (in the marked case) The marked isomorphisms in $\hat \bA^{G}$ are those morphisms which are marked isomorphisms in $\bA$.\qedhere
	\end{enumerate}
\end{ddd}

\begin{ex}\label{giurhgergeg34gergerge}
If $\bA$ is an object of $\preAdd^{(+)}$, then  we will shorten the notation and write
$\widehat{\bA}^{G}$  for $\widehat{\underline{\bA}}^{G}$, where $\underline{\bA}$ is $\bA$ with the trivial $G$-action.

In this case
$\widehat{\bA}^{G}$ is the category of objects of $\bA$ with an action of $G$ by (marked) isomorphisms, and equivariant morphisms. In the marked case, the marked isomorphisms in $\widehat{\bA}^{G}$ are those which are marked in $\bA$.
\end{ex}

Recall  the notation \eqref{f34iuh4fiuhrif894r43r34r3434r34r34r31}   
\begin{theorem}\label{rgier9oger}
We have an equivalence
\[ \lim_{BG}\ell_{\preAdd^{(+)},BG}(\bA)\simeq \ell_{\preAdd^{(+)}}(\hat \bA^{G})\ .\]
\end{theorem}
\ 
\begin{rem}\label{geroigergergerg}
If $\bA$ is a pre-additive category with $G$-action, then
the unmarked version of \cref{rgier9oger} can be obtained from the marked versions by
\begin{eqnarray*} \lim_{BG}\ell_{\preAdd^{(+)},BG}(\bA)&\simeq& \cF_{+}(\ma(\lim_{BG}\ell_{\preAdd^{(+)},BG}\bA)))\\&\simeq& 
 \cF_{+}( \lim_{BG}\ell_{\preAdd^{(+)},BG}(\ma_{BG}(\bA)))
\\&\simeq& \ell_{\preAdd^{(+)}}(\cF_{+}(\widehat{\ma_{BG}(\bA)}^{G}))\end{eqnarray*}
using that $\ma$ (as a right-adjoint, see \eqref{f3rfkj34nfkjf3f3f3f43f}) preserves limits.
Note that
\[ \cF_{+}(\widehat{\ma_{BG}(\bA)}^{G})= \hat \bA^{G}\ ,\]
where on the left-hand side we use \cref{ddd_hat_A_G} in the marked case, and on the right-hand side we use it in the unmarked case. 
 \end{rem}

\begin{rem}
The order of taking the limit $\lim_{BG}$ and the localization  $\ell_{...}$ matters.
For example, consider the additive category $\Mod(\Z)$ with   the trivial $G$-action.
Then
\[ \lim_{BG} \Mod(\Z) \cong\Mod(\Z)\ .\]
On the other hand, $\widehat{\Mod(\Z)}^{G}$ is the category of representations of $G$ on $\Z$-modules.  If $G$ is non-trivial, then it is not equivalent to $\Mod(\Z)$. \end{rem}

For simplicity (and in view of \cref{geroigergergerg}), we   formulate the proof in the marked case, only.
Since the category $\preAdd^{+}$ has a combinatorial model category structure the injective model category structure in $\Fun(BG,\preAdd^{+})$ exists. The proof of this fact involves  Smith's theorem, see e.g. \cite[Thm.~1.7]{beke},  \cite[Sec.~A.2.6 ]{htt}. A textbook reference of the fact stated  precisely in the form we need is  \cite[Prop.~A.2.8.2]{htt}.  

For every fibrant replacement functor $r \colon \id\to R$ in the injective model category structure on
$\Fun(BG,\preAdd^{+})$ we have an equivalence
\begin{equation}\label{eq:fibrant-replacement}
\ell_{\preAdd^{+}}\circ \lim_{ BG}\circ R\simeq \lim_{BG}\circ \ell_{\preAdd^{+},BG}
\end{equation}
of functors from $\Fun(BG,\preAdd^{+})$ to $\preAdd^{+}_{\infty}$ (see e.g. \cite[Prop. 13.5]{bunke} for an argument).
In the following we use the notation  introduced in \cref{gwiogefwerfwefwefewfwadd} and before \cref{gioergegergreg}. Furthermore, we consider the $G$-groupoid $\tilde G$ defined in \cref{rwgioorgfgwergwf}.
We define the functor
 \begin{equation}\label{feoirufeoirfjerfref}
R:=\Fun^{+}_{\preAdd^{+}}(Q(\tilde G),-) \colon \Fun(BG,\preAdd^{+})\to \Fun(BG,\preAdd^{+})
\end{equation} 
together with the natural transformation  $r \colon \id\to R$ induced by $\tilde G\to \Delta^{0}_{\Cat}$ using the canonical isomorphism $\Fun^{+}_{\preAdd^{+}}(Q(\Delta^{0}_{\Cat}),-)\cong \id$.
\begin{lem}\label{wfeioweo9ffwefwf}
	The functor \eqref{feoirufeoirfjerfref} together with the natural transformation  $r$ 
	is a fibrant replacement functor.
\end{lem}
\begin{proof}The morphism $\tilde G\to \Delta^{0}_{\Cat}$ is a non-equivariant equivalence of groupoids. An inverse equivalence is given by any map $\Delta_{\Cat}^{0}\to \tilde G$ classifying some object of $\tilde G$. Since this functor is injective on objects we conclude similarly as for
\cref{lem:power.fib} that the (non-equivariant) morphism $p \colon R(\bA)\to \bA$ it induces   is a weak equivalence. Since $p\circ r=\id$ we conclude that $r \colon \bA\to R(\bA)$ is a  (non-equivariant) weak equivalence, too.  Hence
$r \colon \bA\to R(\bA) $, now considered as a morphism in $\Fun(BG,\preAdd^{+})$, is an  equivalence in the injective model category structure.

	In order to finish the proof we must show that $ R(\bA)$ is fibrant.
	To this end we consider the following square in $\Fun(BG,\preAdd^{+})$, where $c\colon \bC\to \bD$ is a trivial cofibration in $\Fun(BG,\preAdd^{+})$:
	\[\xymatrix{\bC\ar[r]\ar[d]^-{c}&R(\bA)\ar[d]\\\bD\ar[r]\ar@{..>}[ur]&{*}}\]
	We must show the existence of the diagonal lift. 
	
	We use the identification $\Fun_{\preAdd^{+}}^{+}(Q(\tilde G),*)\simeq *$ and  the adjunction of \cref{gioergegergreg}    in order to rewrite the lifting problem as follows.
	\[\xymatrix{\bC\sharp \tilde G\ar[r]^{\phi}\ar[d]& \bA\ar[d]\\\bD\sharp \tilde G\ar[r]\ar@/^1cm/@{..>}[u]^{\tilde d}\ar@{..>}[ur] &{*}}\]
	Since, after forgetting the $G$-action, the  morphism of $c \colon \bC\to \bD$  is a trivial cofibration it is injective on objects.
	We can therefore choose an inverse equivalence $d \colon \bD\to \bC$ (not necessarily $G$-invariant) up to marked isomorphism  {with $d\circ c=\id_\bC$.}
	We can extend the composition
	\[ \bD\xrightarrow{d} \bC\to \bC\times \{1\}\to \bC\sharp \tilde G\] uniquely to a $G$-invariant morphism
	\[ \tilde d \colon \bD\sharp \tilde G\to \bC\sharp \tilde G\]
	by setting
	\[ \tilde d(D,g):=(g(d(g^{-1}D)),g)\ , \quad \tilde d(f \colon D\to D^{\prime},g\to h):=g^{-1} d(g^{-1}f)\sharp (g\to h) \ . \]
	The desired lift can now be obtained as the composition $\phi\circ \tilde d$.
\end{proof}

\begin{proof}[Proof of \cref{rgier9oger}]
By \eqref{eq:fibrant-replacement} and \cref{wfeioweo9ffwefwf}, we have an equivalence
\[ \lim_{BG}\ell_{\preAdd^{+},BG}(\bA)\simeq \ell_{\preAdd^{+}}( \lim_{BG}R(\bA))\ . \]
In order to finish the proof  of  \cref{rgier9oger}, it remains to show that
\[ \lim_{BG}R(\bA)\cong \hat \bA^{G}\ .\]
We define a functor 
\[ \Psi \colon \lim_{BG} R(\bA)=\lim_{BG} \Fun_{\preAdd^{+}}^{+}(Q(\tilde G),\bA)\to \hat \bA^{G} \]
as follows.
\begin{enumerate}
	\item on objects:
	\[ \Psi(\phi):=(\phi(1), (\phi(1 \to g))_{g\in G})\ . \]
	Note that $\phi(g)=g\phi(1)$ by $G$-invariance of $\phi$.
	\item on morphisms:
	\[ \Psi((a_{h})_{h\in \tilde G} \colon \phi\to \psi):=a_{1} \colon \phi(1)\to \psi(1)\ .\]
	One easily checks the relation \ref{guerioggergerg} using that $\phi$ and $\psi$ are $G$-invariant and that $(a_{h})_{h\in \tilde G}$ is a natural transformation.
	\item We observe that $\Psi$ preserves marked isomorphisms.
\end{enumerate}
Finally we check that the
  functor $\Psi$ is an isomorphism of categories. This finishes the proof of \cref{rgier9oger}.
 \end{proof}

 \cref{rgier9oger} implies an analogous statement for additive categories. 

Let $\bA$ be in $\Fun(BG,\preAdd^{(+)})$.
\begin{lem}\label{ergoiejrogregregeg}
	If $\bA$ belongs to the subcategory  $\Fun(BG, \Add^{(+)})$, then $ \hat \bA^{G}$ is a (marked) additive category.
\end{lem}
\begin{proof}
	We must show that $\hat \bA^{G}$ admits finite coproducts. 
	If $(M,\rho)$ and $(M^{\prime},\rho^{\prime})$ are two objects,
	then $(M\oplus M^{\prime},\rho\oplus \rho^{\prime})$ together with the canonical inclusions represents the coproduct of $(M,\rho)$ and $(M^{\prime},\rho^{\prime})$. 
	In the marked case, one furthermore checks by inspection condition \ref{fdblkgjklrgregergergerg} from \cref{reiuheriververvec} for $\bA$ implies this condition for $\hat \bA^{G}$. This condition also implies that $\rho\oplus \rho^{\prime}$ acts by marked isomorphisms as required in the marked case.
\end{proof}
 
  Let  $\bA$ be in $ \Fun(BG, \Add^{+})$.
\begin{kor}
	\label{cor:invariants}
 We have an equivalence
\[ \lim_{BG} \ell_{\Add^{(+)},BG}(\bA)\simeq \ell_{\Add^{(+)}}(\hat \bA^{G})\ .\]
\end{kor}
\begin{proof}
The functor $\cF_{\oplus} \colon \Add^{+}_{\infty}\to \preAdd^{+}_{\infty}$ is a right-adjoint and hence preserves limits. 
Using \cref{rgier9oger}, we obtain equivalences
\begin{align*}
\cF_{\oplus}(\lim_{BG}(\ell_{\Add^{(+)},BG} (\bA))) & \simeq \lim_{BG} \ell_{\preAdd^{(+)},BG} (\cF_{\oplus,BG}(\bA))\\
& \simeq \ell_{\preAdd^{(+)}}(\widehat{\cF_{\oplus,BG}( \bA)}^{G})\\
& \simeq \cF_{\oplus}( \ell_{\Add^{(+)}}(\hat \bA^{G}))
\end{align*}
Since $\hat \bA^{G}$ is additive by \cref{ergoiejrogregregeg},
this implies the assertion by omitting $\cF_{\oplus}$ on both sides.
\end{proof}
 
 \begin{ex}\label{rgvoihifowefwfwe}
 	Let $k$ be a complete normed field and let $\Banach$ denote the category of Banach spaces over $k$ and bounded linear maps.
 This category is additive. Note that only the equivalence class of the norm on an object of $\Banach$ is an invariant of the isomorphism class of the object. We use the norms in order to define a marked pre-additive category $\Banach^{+}$ by marking isometries.  
 
 It is first interesting to observe that $\Banach^{+}$  is not a marked additive category. In fact,
 the Condition \ref{reiuheriververvec}.\ref{fdblkgjklrgregergergerg}  is violated since only the equivalence class of the norm on a direct sum is fixed by the norms on the summands.
 
 We can now calculate the $G$-invariants: 
 By \cref{cor:invariants},
 \[ \lim_{BG} \ell_{\Add,BG} (\underline{\Banach})\simeq \ell_{\Add}(\widehat{\Banach}^{G})\ .\]
 By \cref{giurhgergeg34gergerge}, $\widehat{\Banach}^{G}$
 is the category of Banach-spaces over $k$ with an action by $G$ and equivariant bounded linear maps.
 On the other hand, by \cref{rgier9oger}
 \[ \lim_{BG} \ell_{\preAdd^{+},BG} (\underline{\Banach^{+}})\simeq \ell_{\preAdd^{+}}(\widehat{\Banach^{+}}^{G})\ .\]
 By \cref{giurhgergeg34gergerge}, $\widehat{\Banach^{+}}^{G}$
 is the category of Banach-spaces over $k$ with an  \emph{isometric} action by $G$ and equivariant bounded linear maps which are marked if the are isometric.
 Hence $\cF_{+}( \widehat{\Banach^{+}}^{G} )$ is contained properly in
 $ \widehat{\Banach}^{G} $.
 
 This shows that even if we forget the marking at the end, the marking matters when we form limits.
  \end{ex}
 
 \begin{ex}
 	Let $R$ be a unital ring. We consider the additive categories $\Mod^{?}(R)$ and $\Mod^{?}(R)$, where the decoration $?$ is a condition like \emph{free}, \emph{projection},   \emph{finitely generated} or some
 	combination of these.
 	By \cref{rgier9oger} and \cref{giurhgergeg34gergerge}, we get
 	\[ \lim_{BG}\ell_{\Add,BG}(\underline{\Mod^{?}(R)})\simeq \ell_{\Add}( \Fun(BG,\Mod^{?}(R)))\ .\]
 	
 	Note the difference between limits and colimits: By \cref{rierhigregegerg43t34t34t} we have an equivalence
 	\[ \colim_{BG}\ell_{\Add,BG}(\underline{\Mod^{?}(R)})\simeq \ell_{\Add}(  \Mod^{?}(R[G])) \]
 	for $?=(\fg,\proj), (\fg,\free)$. If $G$ is infinite, then the  interpretation of $?$ on the right-hand side leads to different categories (e.g. finitely free generated $R[G]$-modules are in general not finitely generated $R$-modules with a $G$-action).
 \end{ex}
 
\begin{ex}
	\label{ex:bc}
	For the following example we assume familiarity with equivariant coarse homology theories and the example of equivariant coarse algebraic $K$-homology, see for example \cite[Sec.~2, 3 and 8]{equicoarse}. In particular, recall the definition of the functor $\bV_\bA\colon \BC\to \Add$ of $X$-controlled $\bA$-objects for a bornological coarse space $X$ and an additive category $\bA$ from \cite[Sec.~8.2]{equicoarse}. We define the functor
	\[\bV_\bA^+\colon \BC\to \Add^+\]
	by considering $\bV_\bA$ and marking the $\diag(X)$-controlled isomorphisms.
	
	Let $X$ be a $G$-bornological coarse space and let $\bA$ be an additive category with a $G$-action. By functoriality the marked additive category $\bV_\bA^{+}(X)$ then has an action of $G\times G$. We consider $\bV_\bA^{+}(X)$ as a marked additive category with $G$-action by restricting the $G\times G$ action along the diagonal. As in \cref{ddd_hat_A_G} we can form the category $\widehat{\bV_\bA^{+}}^{G}$.
	
	We define the functor
	\[\bV_\bA^{G}:=\cF_{+}\circ  \widehat{\bV_\bA^{+}}^{G}\colon G\BC\to \Add\ .\]
	One checks that this definition agrees with the definition of $\bV_\bA^G$ from \cite[Sec.~8.2]{equicoarse}.
	
	By definition, equivariant coarse algebraic $K$-homology is the functor $K\bA\cX^G:=K\circ\bV_\bA^{G}$.
	
	The functor $\cF_+\colon \Add^+\to\Add$ descents to a functor $\cF_+\colon \Add^+_\infty\to \Add_\infty$.
	Using \cref{cor:invariants}, we now obtain
	\begin{align*}
	K\bA\cX^G&=K\circ\bV_\bA^{G}=K\circ \cF_{+}\circ  \widehat{\bV_\bA^{+}}^{G}\\
	&\simeq K_\infty\circ \ell_{\Add}\circ \cF_{+}\circ  \widehat{\bV_\bA^{+}}^{G}\\
	&\simeq K_\infty\circ \cF_{+}\circ \ell_{\Add^+}\circ  \widehat{\bV_\bA^{+}}^{G}\\
	&\simeq K_\infty\circ \cF_{+}\circ \lim_{BG}\circ  {\ell_{\Add^+,BG}} \circ \bV_\bA^{+}\ .
	\end{align*}
	This shows that equivariant coarse algebraic $K$-homology can be computed from the non-equivariant version by taking $G$-invariants in marked additive categories.
\end{ex}

In addition to the adjunction \eqref{eq_BG_GOrb_Kan}, for a presentable $\infty$-category $\bC$ we also have an adjunction
\begin{equation}\label{vnvkjenvkjenvkevnkervccerf}
i^{op,*} \colon \Fun(G\Orb^{op},\bC)\leftrightarrows \Fun(BG^{op},\bC) \colon i^{op}_{*}\ . 
\end{equation} In analogy to \eqref{revelrnjjkrnfkjervervverveverv} we consider the functor $C^{G}$ defined as the composition
\[
\mathclap{
\Fun(BG^{op},\preAdd^{(+)}) \xrightarrow{\ell_{\preAdd^{(+)},BG}} \Fun(BG^{op},\preAdd^{(+)}_{\infty})\xrightarrow{i^{op}_{*}} \Fun(G\Orb^{op},\preAdd^{(+)}_{\infty})\ .}
\]
For a (marked) pre-additive category with $G$-action
$\bA$ we are interested in the values $C^{G}(\bA)(G/H)$ for subgroups $H$ of $G$.
\begin{lem}\label{guihiuwfewfwefwefwef}
We have an equivalence
\[ C^{G}(\bA)(G/H)\simeq \ell_{\preAdd^{(+)}}(\widehat{ \Res^{G}_{H}(\bA)}^{H})\ . \]
\end{lem}
\begin{proof}
The argument is very similar to the proof of \cref{vfuiheiwufewfewfefewfewf}. We use  that the induction $S\mapsto G\times_{H}S$
 induces an equivalence
 \[ BH^{op}\simeq (G/H)/i^{op}\ ,\]
 where $(G/H)/i^{op}$ denotes the slice of $i^{op} \colon BG^{op}\to G\Orb^{op}$ under $G/H$.
 Further employing the point-wise formula for the right-Kan extension  functor $i^{op}_{*}$ and the equivalence $BH\simeq BH^{op}$ given by inversion, we get
 \begin{align*}
 C^{G}(\bA)(G/H)&\simeq i^{op}_{*}(\ell_{\preAdd^{(+)},BG^{op}}(  \bA ))(G/H)\\&\simeq
 \lim_{(G/H \to i^{op}(*))\in    (G/H)/i^{op}} \ell_{\preAdd^{(+)},BG^{op}}( \bA )(*)\\
 &\simeq\lim_{  BH^{op}} \ell_{\preAdd^{(+)},BH^{op}}( \Res^{G}_{H}(\bA)  ) \\
 &\simeq \ell_{\preAdd^{(+)}}(\widehat{ \Res^{G}_{H}(\bA)}^{H})\qedhere
  \end{align*}
 \end{proof}

\bibliographystyle{amsalpha}
\bibliography{addcats}
\end{document}